\documentclass[12pt]{article}
\usepackage[margin=1in]{geometry} 

\usepackage{amsthm,amsmath,amssymb}
\usepackage{url} 
\usepackage[authoryear]{natbib}
\usepackage[colorlinks=true,linkcolor=black,citecolor=black,urlcolor=black,breaklinks]{hyperref}
\usepackage[english]{babel}
\usepackage{caption}
\usepackage{subcaption}
\usepackage{graphicx} 
\usepackage{epstopdf}
\usepackage{color}
\usepackage{framed}
\usepackage{lscape}
\usepackage{rotating}
\usepackage{algorithm}
\usepackage[noend]{algpseudocode}
\usepackage{custom_tex}
\usepackage{grffile}
\usepackage{multirow}
\usepackage{xcolor}
\usepackage{setspace}
\usepackage{comment}

\newcommand{\ep}{\varepsilon}

\newcommand{\E}{\mathbb{E}}
\newcommand{\R}{\mathbb{R}}

\newcommand{\supp}{\textnormal{supp}}
\newcommand{\bitem}{\begin{itemize}}
\newcommand{\eitem}{\end{itemize}}
\newcommand{\benum}{\begin{enumerate}}
\newcommand{\eenum}{\end{enumerate}}
\newcommand{\beq}{\begin{equation}}
\newcommand{\eeq}{\end{equation}}
\newcommand{\beqs}{\begin{equation*}}
\newcommand{\eeqs}{\end{equation*}}

\newtheorem{theorem}{Theorem}[section]
\newtheorem{lem}[theorem]{Lemma}

\newtheorem{thm}{Theorem}[section]

\newtheorem{prop}[thm]{Proposition}

\theoremstyle{definition}

\begin{document}
\title{Optimal prediction in the linearly transformed spiked model}
\author{Edgar Dobriban, William Leeb, and Amit Singer}
\maketitle
\abstract{
We consider the \emph{linearly transformed spiked model}, where observations $Y_i$ are noisy linear transforms of unobserved signals of interest $X_i$: 
\begin{align*}
    Y_i = A_i X_i + \varepsilon_i,
\end{align*}
for $i=1,\ldots,n$. The transform matrices $A_i$ are also observed.  We model $X_i$ as random vectors lying on an unknown low-dimensional space.  How should we predict the unobserved signals (regression coefficients) $X_i$?

The naive approach of performing regression for each observation separately is inaccurate due to the large noise. Instead, we develop optimal linear empirical Bayes methods for predicting $X_i$ by ``borrowing strength'' across the different samples. Our methods are applicable to large datasets and rely on weak moment assumptions.  The analysis is based on random matrix theory.

We discuss applications to signal processing, deconvolution, cryo-electron microscopy, and missing data in the high-noise regime. For missing data, we show in simulations that our methods are faster, more robust to noise and to unequal sampling than well-known matrix completion methods.

}

\section{Introduction}

In this paper we study the \emph{linearly transformed spiked model}, where the observed data vectors $Y_i$ are noisy linear transforms of unobserved signals of interest $X_i$:
$$
    Y_i = A_i X_i + \varepsilon_i, \,\,\, i=1,\ldots,n.
$$
We also observe the transform matrices $A_i$. A transform matrix reduces the dimension of the signal $X_i \in \R^p$ to a possibly observation-dependent dimension $q_i \le p$, thus $A_i \in \mathbb{R}^{q_i \times p}$. Moreover, the signals are assumed to be random vectors lying on an unknown low-dimensional space, an assumption sometimes known as a spiked model \citep{johnstone2001distribution}.

Our main goal is to recover (estimate or predict) the unobserved signals $X_i$. The problem arises in many applications, some of which are discussed in the next section. Recovery is challenging due to the two different sources of information loss: First, the transform matrices $A_i$ reduce the dimension, since they are generally not invertible. It is crucial that the transform matrices differ between observations, as this allows us to reconstruct this lost information from different ``snapshots'' of $X_i$. Second, the observations are contaminated with additive noise $\ep_i$. We study the regime where the size of the noise is much larger than the size of the signal. This necessitates methods that are not only numerically stable, but also reduce the noise significantly.

This setup can be viewed as a different linear regression problem for each sample $i=1,\ldots,n$, with outcome vector $Y_i$ and covariate matrix $A_i$. The goal is then to estimate the regression coefficients $X_i$. Since $X_i$ are random, this is also a random effects model. Our specific setting, with low-rank $X_i$, is more commonly considered in spiked models, and we will call $X_i$ the \emph{signals}.

This paper assumes that the matrices $A_i^\top A_i \in \mathbb{R}^p$ are diagonal. Equivalently, we assume that the matrices $A_i^\top A_i$ all commute (and so can be jointly diagonalized). We will refer to this as the \emph{commutative} model. This is mainly a technical assumption and we will see that it holds in many applications.

With large noise, predicting one $X_i$ using one $Y_i$ alone has low accuracy. Instead, our methods predict $X_i$ by ``borrowing strength'' across the different samples. For this we model $X_i$ as random vectors lying on an unknown low-dimensional space, which is reasonable in many applications. Thus our methods are a type of empirical Bayes methods \citep{efron2012large}.

Our methods are fast and applicable to big data, rely on weak distributional assumptions (only using moments), are robust to high levels of noise, and have certain statistical optimality results. Our analysis is based on recent insights from random matrix theory, a rapidly developing area of mathematics with many applications to statistics \citep[e.g.,][]{bai2009spectral, paul2014random, yao2015large}.

\subsection{Motivation}
\label{sec-motivation}

We study the linearly transformed model motivated by its wide applicability to several important data analysis scenarios.

\subsubsection{PCA and spiked model}

In the well-known \emph{spiked model} one observes data $Y_i$ of the form $Y_i  = X_i + \ep_i$, where $X_i \in \R^p $ are unobserved signals lying on an unknown low dimensional space, and $\ep_i  \in \R^p $ is noise. With $A_i=I_p$ for all $i$, this is a special case of the commutative linearly transformed spiked model. 

The spiked model is fundamental for understanding principal component analysis (PCA), and has been thoroughly studied under high-dimensional asymptotics. Its understanding will serve as a baseline in our study. Among the many references, see for instance \cite{johnstone2001distribution, baik2005phase, baik2006eigenvalues, paul2007asymptotics, nadakuditi2008sample, nadler2008finite, bai2012estimation, bai2012sample, benaych2012singular, onatski2012asymptotics, onatski2013asymptotic, donoho2013optimal, onatski2014signal, nadakuditi2014optshrink, gavish2014optimal, johnstone2015testing, hachem2015survey}.

\subsubsection{Noisy deconvolution in signal processing}

The transformed spiked model is broadly relevant in signal acquisition and imaging. Measurement and imaging devices nearly never measure the ``true'' values of a signal. Rather, they measure a weighted average of the signal over a small window in time and/or space. Often, this local averaging can be modeled as the application of a convolution filter. For example, any time-invariant recording device in signal processing is modeled by a convolution \citep{mallat2008wavelet}. Similarly, the blur induced by an imaging device can be modeled as convolution with a function, such as a Gaussian \citep{blackledge2006digital,campisi2016blind}. In general, this filter will not be numerically invertible.

As is well-known, any convolution filter $A_i$ is linear and diagonal in the Fourier basis; for example, see \cite{stein2011fourier}. Consequently, $A_i^\top A_i$ is also diagonalized by the Fourier basis. Convolutions thus provide a rich source of examples of the linearly transformed spiked model.

\subsubsection{Cryo-electron microscopy (cryo-EM)}

Cryo - electron microscopy (cryo-EM) is an experimental method for mapping the structure of molecules. It allows imaging of heterogeneous samples, with mixtures  or multiple conformations of molecules. This method has received a great deal of recent interest, and has recently led to the successful mapping of important molecules \citep[e.g.,][]{bai2015cryo, callaway2015revolution}. 

Cryo-EM works by rapidly freezing a collection of molecules in a layer of thin ice, and firing an electron beam through the ice to produce two-dimensional images. The resulting observations can be modeled as $Y_i = A_i X_i + \ep_i$, where $X_i$ represents an unknown 3D molecule; $A_i$ randomly rotates the molecule, projects it onto the xy-plane, and applies blur to the resulting image; and $\ep_i$ is noise \citep{katsevich2015covariance}. Since a low electron dose is used to avoid destroying the molecule, the images are typically very noisy.

When all the molecules in the batch are identical, i.e.\ $X_i = X$ for all $i$, the task of \emph{ab-initio 3D reconstruction} is to recover the 3D molecule $X$ from the noisy and blurred projections $Y_i$ \citep{Kam1980}. Even more challenging is the problem of \textit{heterogeneity}, in which several different molecules, or one molecule in different conformations, are observed together, without labels. The unseen molecules can usually be assumed to lie on some unknown low-dimensional space  \citep{katsevich2015covariance, anden2015covariance}. Cryo-EM observations thus fit the linearly transformed spiked model.

The noisy deconvolution problem mentioned above is also encountered in cryo-EM. The operators $A_i$ induce blur by convolution with a point-spread function (PSF), thus denoising leads to improved 3D reconstruction \citep{bhamre2016denoising}. The Fourier transform of the point-spread function is called the \emph{contrast transfer function (CTF)}, and the problem of removing its effects from an image is known as \emph{CTF correction}.

\subsubsection{Missing data}

Missing data can be modeled by \emph{coordinate selection operators} $A_i$, such that $A_i(k,l) = 1$ if the $k$-th coordinate selected by $A_i$ is $l$, and $A_i(k,l) = 0$ otherwise. Thus $A_i^\top A_i$ are diagonal with 0/1 entries indicating missing/observed coordinates.  In the low-noise regime, missing data in matrices has recently been studied under the name of \emph{matrix completion} \citep[e.g.,][]{candes2009exact, candes2010power, keshavan2009matrix, keshavan2010matrix, koltchinskii2011nuclear, negahban2011estimation, recht2011simpler, rohde2011estimation, jain2013}. As we discuss later, our methods perform well in the high-noise setting of this problem.

\subsection{Our contributions}

Our main contribution is to develop general methods predicting $X_i$ in linearly transformed spiked models $Y_i = A_i X_i + \ep_i$. We develop methods that are fast and applicable to big data, rely on weak moment assumptions, are robust to high levels of noise, and have certain optimality properties.

Our general approach is as follows:  We model $X_i$ as random vectors lying on an unknown low-dimensional space, $X_i = \sum_{k=1}^r \ell_k^{1/2} z_{ik} u_k$ for fixed unit vectors $u_k$ and mean-zero scalar random variables $z_{ik}$, as usual in spiked models. In this model, the Best Linear Predictor (BLP), also known as the Best Linear Unbiased Predictor (BLUP), of $X_i$ given $Y_i$ is well known \citep{searle2009variance}. (The more well known Best Linear Unbiased Estimator (BLUE) is defined for fixed-effects models where $X_i$ are non-random parameters.) The BLP depends on the unknown population principal components $u_k$. In addition, it has a complicated form involving matrix inversion.

Our contributions are then: 

\benum
\item We show that the BLP reduces to a simpler form in a certain natural high-dimensional model where $n,p\to\infty$ such that $p/n\to\gamma>0$ (Sec.\ \ref{unif_mod}). In this simpler form, we can estimate the population principal components using the principal components (PCs) of the \emph{backprojected data} $A_i^\top Y_i$ to obtain an Empirical BLP (EBLP) predictor (a type of moment-based empirical Bayes method), known up to some scaling coefficients. By an exchangeability argument, we show that the optimal scaling coefficients are the same as optimal singular value shrinkage coefficients for a certain novel random matrix model (Sec.\ \ref{red_sv}). 

\item We derive the asymptotically optimal singular value shrinkage coefficients (Sec.\ \ref{deriv_opt}), by characterizing the spectrum of the backprojected data matrix (Sec.\ \ref{spec_bp}). This is our main technical contribution.

 
\item We derive a suitable ``normalization'' method to make our method fully implementable in practice (Sec.\ \ref{Normalization}). This allows us to estimate the optimal shrinkage coefficients consistently, and to use well-known optimal shrinkage methods \citep{nadakuditi2014optshrink, gavish2014optimal}. We also discuss how to estimate the rank (Sec. \ref{sec-rk}).

\item We also solve the out-of-sample prediction problem, where new $Y_0,A_0$ are observed, and $X_0$ is predicted using the existing data (Sec.\ \ref{oos_pred_alg}). 

\item We compare our methods to existing approaches for the special case of missing data problems via simulations (Sec.\ \ref{mx_comp}). These are reproducible with code provided on Github at \url{https://github.com/wleeb/opt-pred}.
\eenum


%
%
%
%

\section{Empirical linear prediction}
\label{emp_lin}

\subsection{The method}
\label{in_sample_alg}

Our method is simple to state using elementary linear algebra. We give the steps here for convenience. In subsequent sections, we will explain each step, and prove the optimality of this procedure over a certain class of predictors. Our method has the following steps:

\begin{enumerate}

\item \emph{Input}: Noisy linearly transformed observations $Y_i$, and transform matrices $A_i$, for $i=1,\ldots,n$. Preliminary rank estimate $r$ (see Sec. \ref{sec-rk} for discussion).

\item Form backprojected data matrix $B = [A_1^\top Y_1,\dots,A_n^\top Y_n]^\top$ and diagonal normalization matrix $\hat{M} = n^{-1/2} \sum_{i=1}^n A_i^\top A_i$. Form the normalized, backprojected data matrix $\tilde{B} = B \hat{M}^{-1}$.

\item \label{step-whiten}

 \emph{(Optional)}
Multiply $\tilde{B}$ by a diagonal whitening matrix $W$, \mbox{$\tilde{B} \leftarrow \tilde{B}W$}. The definition of $W$ is given in Sec.\ \ref{sec-est-W}.

\item Compute the singular values  $\sigma_k$ and the top $r$ singular vectors $ \hat u_k, \hat v_k$ of the matrix $\tilde B$.

\item Compute $\hat X = (\hat X_1,\ldots,\hat X_n)^\top = \sum_{k=1}^r \hat \lambda_k  \hat u_k \hat{v}_k^\top$.

Here $\hat\lambda_k$ are computed according to Sec.\ \ref{deriv_opt}: $\hat\lambda_k = \hat \ell_k^{1/2} \hat c_k \hat{\tilde{c}}_k$, where $\hat \ell_k, \hat c_k, \hat{\tilde{c}}_k$ are estimated based on the formulas given in Theorem \ref{spec} by plug-in. Specifically, $\hat \ell_k = 1/\hat D(\sigma_k^2)$,  $\hat c_k^2 = \hat m(\sigma_k^2) / [\hat D'(\sigma_k^2) \hat \ell_k]$, $ \hat{\tilde{c}}_k^2 = \hat{ \underline {m}}(\sigma_k^2) / [\hat D'(\sigma_k^2) \hat \ell_k]$, where $\hat m, \hat{ \underline {m}}, \hat D, \hat D'$ are the plug-in estimators of the Stieltjes-transform-like functionals of the spectral distribution, using the bottom $\min(n,p)-r$ eigenvalues of the sample covariance matrix of the backprojected data. For instance, $\hat m$ is given in equation \eqref{st_plugin} (assuming $p\le n$): 
$$
\hat{m}(x) = \frac{1}{p-r}\sum_{k=r+1}^p \frac{1}{\sigma_k^2 -x}.
$$

\item If whitening was performed (Step \ref{step-whiten}), unwhiten the data, \mbox{$\hat{X} \leftarrow \hat{X}W^{-1}$}.

\item \emph{Output}: Predictions $\hat X_i$ for $X_i$,  for $i=1,\ldots,n$. 

\end{enumerate}

The complexity of the method is dominated by computing the singular value spectrum of the backprojected matrix, which takes $O(\min(n,p)^2 \cdot \max(n,p))$ floating point operations.
As we will show in Sec.\ \ref{W_norm}, by choosing a certain whitening matrix $W$, the algorithm will only require computing the top $r$ singular vectors and values of the backprojected data matrix, and so can typically be performed at an even lower cost using, for example, the Lanczos algorithm \citep{golub2012matrix}, especially when there is a low cost of applying the matrix $\tilde{B}$ to a vector.

\subsection{Motivation I: from BLP to EBLP}

We now explain the steps of our method. We will use the mean-squared error $\mathbb{E}\|\hat{X}_i - X_i\|^2$ to assess the quality of a predictor $\hat{X}_i$. Recall that we modeled the signals as $X_i = \sum_{k=1}^r \ell_k^{1/2} z_{ik} u_k$. It is well known in random effects models \citep[e.g.,][]{searle2009variance} that the best linear predictor, or BLP, of one signal $X_i$ using $Y_i$, is:
\beq
\label{blp}
    \hat{X}_i^{BLP} = \Sigma_X A_i^\top 
        (A_i \Sigma_X A_i^\top + \Sigma_{\varepsilon})^{-1} Y_i.
\eeq
Here, $\Sigma_X= \sum_{k=1}^r \ell_k u_k u_k^\top $ denotes the covariance matrix of one $X_i$, and $\Sigma_\varepsilon$ is the covariance matrix of the noise $\varepsilon_i$. These are unknown parameters, so we need to estimate them in order to get a \emph{bona fide} predictor. Moreover, though $A_i$ are fixed parameters here, we will take them to be random later.

We are interested in the ``high-dimensional'' asymptotic regime, where the dimension $p$ grows proportionally to the number of samples $n$; that is, $p = p(n)$ and $\lim_{n \to \infty} p(n) / n = \gamma > 0$. 
In this setting it is in general not possible to estimate the population covariance $\Sigma_X$ consistently. Therefore, we focus our attention on alternate methods derived from the BLP.

The BLP involves the inverse of a matrix, which makes it hard to analyze.
However,  for certain \emph{uniform models} (see Sec.\ \ref{unif_mod} for a precise definition), we can show that the BLP is asymptotically equivalent to a simpler linear predictor not involving a matrix inverse: 

$$
    \hat{X}_i^{0} = \sum_{k=1}^r \eta_k^{0} \langle A_i^\top Y_i,u_k\rangle u_k.
$$
Here $\eta_k^0$ are certain constants given in Sec.\ \ref{unif_mod}. This simple form of the BLP       
%
%
%
guides our choice of predictor when the true PCs are not known. Let $\hat{u}_1,\dots,\hat{u}_r$ be the empirical PCs; that is, the top eigenvectors of the sample covariance $\sum_{i=1}^n (A_i^\top Y_i)(A_i^\top Y_i)^\top / n$, or equivalently, the top left singular vectors of the matrix $[A_1^\top Y_1,\dots,A_n^\top Y_n]^\top$. For coefficients $\eta = (\eta_1,\dots,\eta_r)$, substituting $\hat{u}_k$ for $u_k$ leads us to the following \emph{empirical linear predictor}:
$$
    \hat{X}_i^\eta = \sum_{k=1}^r \eta_k \langle A_i^\top Y_i,\hat{u}_k\rangle \hat{u}_k.
$$

Note that, since the empirical PCs $\hat{u}_k$ are used in place of the population PCs $u_k$, the coefficients $\eta_k$ defining the BLP are no longer optimal, and must be adjusted downwards to account for the non-zero angle between $u_k$ and $\hat{u}_k$. This phenomenon was studied in the context of the ordinary spiked model in \cite{singer2013two}.

\subsection{Motivation II: Singular value shrinkage}
\label{red_sv}

Starting with BLP and replacing the unknown population PCs $u_k$ with their empirical counterparts $\hat{u}_k$, we were lead to a predictor of the form $ \hat{X}_i^\eta = \sum_{k=1}^r  \eta_k \langle B_i , \hat{u}_k \rangle \hat{u}_k$, where $ B_i  = A_i^\top Y_i$ are the backprojected data. Now, the matrix $\hat{X}^\eta = [\hat{X}_1^\eta,\dots,\hat{X}_n^\eta]^\top$ has the form
\begin{align}
\label{X_eta}
    \hat{X}^\eta = \sum_{k=1}^r \eta_k \cdot B\hat{u}_k \hat{u}_k^\top 
                 = \sum_{k=1}^r \eta_k\sigma_k(B)  \cdot \hat{v}_k\hat u_k ^\top.
\end{align}
This has the same singular vectors as the matrix $B = [B_1,\dots,B_n]^\top$ of backprojected data. 

From now on, we will consider the $A_i$ as random variables, which corresponds to an average-case analysis over their variability. Then observe that the predictors $\hat X_i^\eta$ are exchangeable random variables with respect to the randomness in $A_i,\ep_i$, because they depend symmetrically on the data matrix $B$. Therefore, the prediction error for a sample equals the average prediction error over all $X_i$, which is the normalized Frobenius norm for predicting the matrix $X = (X_1,\ldots,X_n)^\top$:
$$
    \mathbb{E}\| \hat{X}_i^{\eta} - X_i \|^2 
    = \frac{1}{n} \mathbb{E} \| \hat{X}^\eta - X\|_F^2.
$$

Therefore, the empirical linear predictors are equivalent to performing singular value shrinkage of the matrix $B$ to estimate $X$.  That is, singular value shrinkage predictors are in one-to-one correspondence with the in-sample empirical linear predictors. 
%
%
Because singular value shrinkage is minimax optimal for matrix denoising problems with Gaussian white noise \citep{donoho2014minimax}, it is a natural choice of predictor in the more general setting we consider in this paper, where an optimal denoiser is not known.


%
%
%

\subsection{The class of predictors: shrinkers of normalized, backprojected data}
\label{Normalization}

Motivated by the previous two sections, we are led to singular value shrinkage predictors of the matrix $X$. However, it turns out that rather than shrink the singular values of the matrix $B$ of backprojected data $A_i^\top Y_i$, it is more natural to work instead with the matrix $\tilde{B}$ with rows $\tilde{B}_i = M^{-1} A_i^\top Y_i$, where $M = \E A_i^\top A_i$ is a diagonal normalization matrix. We will show later that we can use a sample estimate of $M$.

The heuristic to explain this is that we can write $A_i^\top A_i = M + E_i$, where $E_i$ is a mean zero diagonal matrix. We will show in the proof of Thm.\ \ref{spec} that because the matrices $A_i^\top A_i$ commute, the matrix with rows $E_i X_i / \sqrt{n}$ has operator norm that vanishes in the high-dimensional limit $p/n \to \gamma$. Consequently, we can write:
\begin{align*}
    B_i = A_i^\top Y_i = M X_i + A_i^\top \varepsilon_i + E_i X_i 
                 \sim \underbrace{M X_i}_{signal} 
                        +  \underbrace{A_i^\top \varepsilon_i}_{noise}
\end{align*}
Since $X_i$ lies in an $r$-dimensional subspace, spanned by $u_1,\dots,u_r$, $M X_i$ also lies in the $r$-dimensional subspace spanned by $M u_1,\dots,M u_r$. Furthermore, $A_i^\top \varepsilon_i$ is mean-zero and independent of $M X_i$. Consequently, $A_i^\top Y_i$ looks like a spiked model, with signal $M X_i$ and noise $A_i^\top \ep_i$.

Shrinkage of this matrix will produce a predictor of $M X_i$, not $X_i$ itself. However, multiplying the data by $M^{-1}$ fixes this problem: we obtain the approximation:
\begin{align*}
    \tilde{B}_i = M^{-1} A_i^\top Y_i
                 \sim X_i +  \underbrace{M^{-1} A_i^\top \varepsilon_i}_{noise}.
\end{align*}
After this normalization, the target signal of any shrinker becomes the true signal $X_i$ itself.

Motivated by these considerations, we can finally state the class of problems we study. We consider predictors of the form:
\begin{align*}
    \hat X_i^\eta = \sum_{k=1}^r \eta_k \langle \tilde B_i , \hat u_k \rangle \hat u_k 
\end{align*}
where $\tilde{B}_i = M^{-1} A_i^\top Y_i$, and we seek the AMSE-optimal coefficients $\eta_k^*$ in the high-dimensional limit $p/n \to \gamma$; that is, our goal is to find the optimal coefficients $\eta_k$, minimizing the AMSE:
$$
    \eta^* = \operatorname*{arg\,min}_\eta 
             \lim_{p,n\to\infty}\mathbb{E} \|\hat{X}_i^\eta - X_i\|^2.
$$
We will show that the limit exists. The corresponding estimator $\hat{X}_i^{\eta^*}$ will be called the \emph{empirical best linear predictor (EBLP)}. We will: (1) show that it is well-defined; (2) derive the optimal choice of $\eta_k$; (3) derive consistent estimators of the optimal $\eta_k$; and (4) derive consistently estimable formulas for the AMSE. As before, finding the optimal $\eta_k$ is equivalent to performing optimal singular value shrinkage on the matrix $\tilde{B} = [\tilde{B}_1,\dots,\tilde{B}_n]^\top$.

\section{Derivation of the optimal coefficients}
\label{deriv_opt}

As described in Sec.\ \ref{emp_lin}, we wish to find the AMSE-optimal coefficients $\eta_k$ for predictors of the form $ \hat{X}_i^\eta = \sum_{k=1}^r  \eta_k \langle \tilde{B}_i , \hat{u}_k \rangle \hat{u}_k$, where $\tilde{B}_i = M^{-1} A_i^\top Y_i$ is the normalized, backprojected data. Equivalently, we find the optimal singular values of the matrix with the same singular vectors as $\tilde{B} = [\tilde{B}_1,\dots,\tilde{B}_n]^\top$.

Singular value shrinkage has been the subject of a lot of recent research. It is now well known that optimal singular value shrinkage depends on the asymptotic spectrum of the data matrix $\tilde{B}$ \cite[e.g.,][]{nadakuditi2014optshrink,gavish-donoho-2017}. We now fully characterize the spectrum, and use it to derive the optimal singular values. We then show that by estimating the optimal singular values by plug-in, we get the method described in Sec.\ \ref{in_sample_alg}.

\subsection{The asymptotic spectral theory of the back-projected data}
\label{spec_bp}
The main theorem characterizes the asymptotic spectral theory of the normalized backprojected data matrix $\tilde B = B M^{-1}$, and of the unnormalized version $B = [A_1^\top Y_1,\dots,A_n^\top Y_n]^\top$. Our data are iid samples of the form $Y_i = A_i X_i + \varepsilon_i$.

We assume that the signals have the form $X_i= \sum_{k=1}^r \ell_k^{1/2} z_{ik} u_k$. Here $u_k$ are deterministic signal directions with $\|u_k\|=1$. We will assume that $u_k$ are delocalized, so that $|u_{k}|_{\infty} \le C_p$ for some constants $C_p\to 0$ that we will specify later.  The scalars $z_{ik}$ are standardized independent random variables, specifying the variation in signal strength from sample to sample. For simplicity we assume that the deterministic spike strengths are different and sorted: $\ell_1>\ell_2>\ldots>\ell_r>0$. 

For a distribution $H$, let $F_{\gamma,H}$ denote the generalized Marchenko-Pastur distribution induced by $H$ with aspect ratio $\gamma$ \citep{marchenko1967distribution}. Closely related to $F_{\gamma,H}$ is the so-called \emph{companion distribution} $\underline F_{\gamma,H}(x)  = \gamma F_{\gamma,H}(x)  + (1-\gamma)\delta_0$. We will also need the Stieltjes transform $m_{\gamma,H}$ of $F_{\gamma,H}$, $m_{\gamma,H}(z)  = \int (x-z)^{-1} d F_{\gamma,H}(x)$, and the Stieltjes transform $\underline m_{\gamma,H}$ of $\underline F_{\gamma,H}$.  Based on these, one can define the D-transform of $F_{\gamma,H}$ by 
\[D_{\gamma,H}(x) = x\cdot m_{\gamma,H}(x) \cdot \underline m_{\gamma,H}(x).\] 
Up to the change of variables $x = y^2$, this agrees with the D-transform defined in \cite{benaych2012singular}.  Let $b^2:=b_H^2$ be the supremum of the support of $F_{\gamma,H}$, and $D_{\gamma,H}(b_H^2) = \lim_{t\downarrow b} D_{\gamma,H}(t^2)$. It is easy to see that this limit is well defined, and is either finite or $+\infty$.

We will assume the following conditions:
\begin{enumerate}

\item {\bf Commutativity condition}. \label{aaa111}
The matrices $A_i^\top A_i$ commute with each other. Equivalently, they are jointly diagonal in some known basis. For simplicity of notation, we will assume without loss of generality that the $A_i^\top A_i$ are diagonal.


\item {\bf Backprojected noise}. \label{aaa222}
The vectors $\ep_i^*= A_i^\top \varepsilon_i$ have independent entries of mean zero. If $H_p$ is the distribution function of the variances of the entries of $M^{-1} \ep_i^*$, then $H_p$ is bounded away from zero; and $H_p \Rightarrow H$ almost surely, where $H$ is a compactly supported distribution.

\item {\bf Maximal noise variance}.
The supremum of the support of $H_p$ converges almost surely to the upper edge of the support of $H$.

\item {\bf Noise moments}.
$\mathbb{E}|\ep^*_{ij}|^{6+\phi}<C$, $\mathbb{E} |E_{ij}|^{6+\phi}<C$ (recall that we defined $E_i = A_i^\top A_i - M$).

\item {\bf Signal.}\label{aaa0}
 One of the following two assumptions holds for the signal directions $u_k$ and signal coefficients $z_{ij}$: 

\begin{itemize}
\item {\bf Polynomial moments and delocalization}. Suppose $\E|z_{ij}|^m \le C < \infty$ for some $m>4$ and for all $k$
$$\| u_k\|_{\infty} \cdot p^{(2+c)/m} \to_{a.s} 0$$
for some $c>0$.
\item {\bf Exponential moments and logarithmic delocalization}. Suppose the $z_{ij}$ are sub-gaussian in the sense that $\E \exp(t |z_{ij}|^2) \le C$ for some $t>0$ and  $C < \infty$, and that  for all $k$
$$\| u_k\|_{\infty} \cdot \sqrt{\log p} \to_{a.s} 0. $$
\end{itemize}

\item {\bf Generic signal.}
\label{aaa}
Let $P$ be the diagonal matrix with $P_{jj} = \text{Var}[M_j^{-1} \ep^*_{ij}]$, where $M_j$ are the diagonal entries of the diagonal matrix  $M = \E A_i^\top A_i$. Then $u_j$ are \emph{generic} with respect to $P$, in the sense that there are some constants $\tau_k>0$ such that:
$$
    u_j^\top (P - zI_p)^{-1} u_k \to I(j = k)\cdot \tau_k\cdot m_H(z)
$$
for all $z \in \mathbb{C}^+$.

\end{enumerate}

Before stating the main results, we make a few remarks on these assumptions. Assumption \ref{aaa111} holds for many applications, as discussed in Sec.\ \ref{sec-motivation}. However, our analysis will go through if a weaker condition is placed on matrices $A_i^\top A_i$, namely that they are \emph{diagonally dominant} in a known basis, in the sense that the off-diagonal elements are asymptotically negligible to the operator norm. Because it does not change anything essential in the analysis, for ease of exposition we will analyze the exact commutativity condition.

\begin{figure}[h]
\centering
  \includegraphics[scale=0.33]{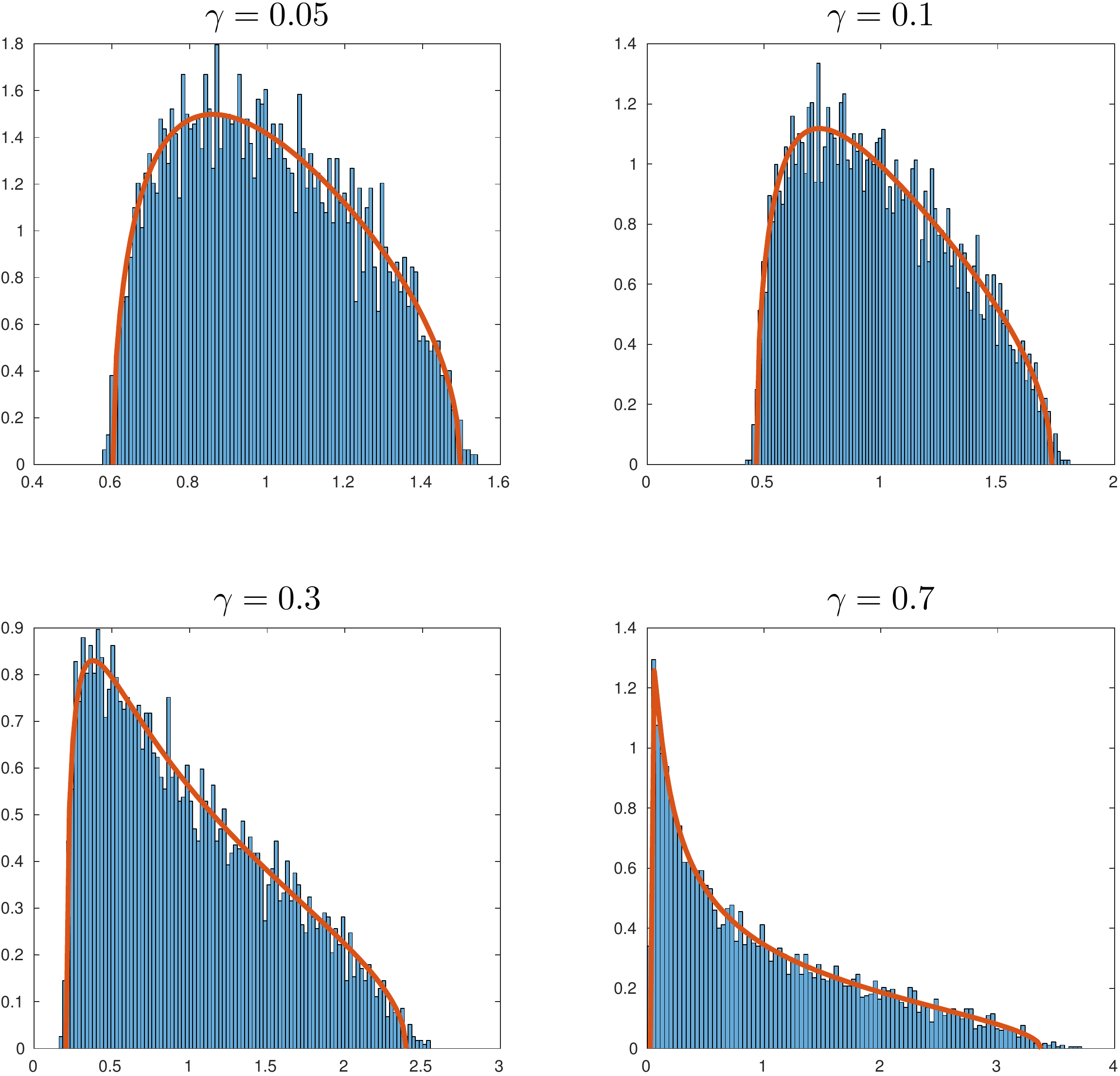}
\caption{
Histograms of empirical eigenvalues of whitened, backprojected noise using 30 CTFs, plotted against the Marchenko-Pastur density for different aspect ratios $\gamma$.
}
\label{fig-mp-cryo}
\end{figure}

The part of Assumption \ref{aaa222} that the entries of $\ep_i^*= A_i^\top \varepsilon_i$ are independent is easily checked for certain problems, such as missing data with independently selected coordinates. However, it may not always hold. For example, in the problem of CTF correction in cryo-EM (see Sec.\ \ref{sec-motivation}), each $A_i$ may be one of a discrete number of different CTFs; in this case, the assumption will not hold exactly. However, we have found in practice that the Marchenko-Pastur law holds even in this regime. To illustrate this, in Fig.\ \ref{fig-mp-cryo} we plot histograms of the sample covariance eigenvalues of simulated backprojected isotropic Gaussian noise using 30 different synthetic CTFs, generated using the ASPIRE software package \citep{aspire}, for 30 defocus values between 0.5 and 3. We plot the coefficients of the backprojected noise in the first frequency block of a steerable basis with radial part the Bessel functions, as described in \cite{bhamre2016denoising} and \cite{zhao2016fast}. Because this frequency block only contains 49 coefficients, the histogram we plot is for 100 draws of the noise. We whiten the backprojected noise, so the population covariance is the identity. As is evident from the figure, there is a very tight agreement between the empirical distribution of eigenvalues and the Marchenko-Pastur laws.

Assumption \ref{aaa0} about the signals presents a tradeoff between the delocalization of the spike eigenvectors and the moments of the signal coefficients. If a weak polynomial moment assumption or order $m$ holds for the signal coefficients $z_{ij}$, then it requires a delocalization at a polynomial rate $p^{-(2+c)/m}$ for the spike eigenvectors. In particular, this implies that at least a polynomial number of coefficients of $u_k$ must be nonzero, so that $u_k$ must be quite non-sparse. In contrast, if we assume a stronger sub-Gaussian moment condition for the noise, then only a logarithmic delocalization is required, which allows $u_k$ to be quite sparse. 

This assumption is similar to the incoherence condition from early works on matrix completion \citep[e.g.,][etc.]{candes2009exact}. Later works have shown that some form of recovery is possible even if we do not have incoherence \citep[e.g.,][]{koltchinskii2011nuclear}. However, in our case, complete sparsity of order one (i.e., only a fixed number of nonzero coordinates) seems impossible to recover. Indeed, suppose the rank is one and $u = (1,0,\ldots,0)$. Then, all information about $u$ and $z$ is in the first coordinate. In our sampling model, we observe a fixed fraction $q$ of the coordinates, and we can have $q<1$. Thus, for the unobserved coordinates, there is no information about the $z_i$. Therefore,  with the current random sampling mechanism, we think that accurate estimation is not possible for fixed sparsity. 

Assumption \ref{aaa} generalizes the existing conditions for spiked models. In particular, it is easy to see that it holds when the vectors $u_k$ are random with independent coordinates.  Specifically, let $x$ be a random vector with iid zero-mean entries with variance $1/p$. Then $\mathbb{E} x^\top (P - zI_p)^{-1} x = p^{-1}\text{tr} (P - zI_p)^{-1}$. Assumption \ref{aaa} requires that this converges to  $m_H(z)$, which follows from $H_p \Rightarrow H$. However, Assumption \ref{aaa} is more general, as it does not require any kind of randomness in $u_k$.

Our main result in this section is the following.

\begin{theorem}[Spectrum of transformed spiked models]
\label{spec}

Under the above conditions, the eigenvalue distribution of $\tilde B^\top \tilde B / n$ converges to the general Marchenko-Pastur law $F_{\gamma,H}$ a.s. In addition, for $k \le r$, the $k$-th largest eigenvalue of $\tilde B^\top \tilde B/ n$ converges, $\lambda_k(\tilde B^\top \tilde B)/ n \to t^2_k$ a.s., where
\begin{equation}
\label{sv_eq}
    t_k^2=
    \left\{
	\begin{array}{ll}
    D_{\gamma,H}^{-1}(\frac{1}{\ell_k}) 
        & \mbox{\, if \, }  \ell_k>1/D_{\gamma,H}(b_H^2), \\
    b_H^2 
        & \mbox{\, otherwise.}
	\end{array}
    \right.
\end{equation}

Moreover,  let $\hat u_k$ be the right singular vector of $\tilde B$ corresponding to $\lambda_k(\tilde B^\top \tilde B)$. Then $(u_j^\top \hat u_k)^2 \to c_{jk}^2$ a.s., where
\begin{equation}
\label{cos_inn22}
    c_{jk}^2=
    \left\{
    \begin{array}{ll}
    \frac{m_{\gamma,H}(t_k^2)}{D_{\gamma,H}'(t_k^2) \ell_k} 
        & \mbox{\, if \, } j=k 
        \mbox{\, and \, } \ell_k>1/D_{\gamma,H}(b_H^2), \\
    0 
        &  \mbox{\, otherwise.}
    \end{array}
    \right.
\end{equation}

Finally, let $Z_j = n^{-1/2}(z_{1j},\ldots,z_{nj})^\top$, and let $\hat{Z}_k$ be the $k$-th left singular vector of $\tilde B$. Then $(Z_j^\top \hat{Z}_k)^2 \to \tilde c_{jk}^2$ a.s., where
\begin{equation}
\label{cos_out44}
    \tilde{c}_{jk}^2=
    \left\{
    \begin{array}{ll}
    \frac{\underline{m}_{\gamma,H}(t_k^2)}{D_{\gamma,H}'(t_k^2)\ell_k} 
        & \mbox{\, if \, } j=k 
        \mbox{\, and \, }  \ell_k>1/D_{\gamma,H}(b_H^2), \\
    0 
        &  \mbox{\, otherwise.}
    \end{array}
    \right.
\end{equation}
\end{theorem}

The proof is in Sec.\ \ref{spec_pf}. While the conclusion of this theorem is very similar to the results of \cite{benaych2012singular}, our observation model $Y_i = A_iX_i+\ep_i$ is entirely different from the one in that paper; we are addressing a different problem.  Moreover, our technical assumptions are also more general and more realistic, and only require finite moments up to the sixth moment, unlike the more stringent conditions in previous work. In addition, we also have the result below, which differs from existing work. 

For the un-normalized backprojected matrix $B$, a version of Thm.\ \ref{spec} applies \emph{mutatis mutandis}. Specifically, we let $H_p$ be the distribution of the variances of $A_i^\top \ep_i$. We replace $I_p$ with $M$ in the assumptions when needed, so we let $\tau_k = \lim_{n\to\infty}\|M u_k\|^2$, and $\nu_j=M u_j/\|M u_j\|$. Then the above result holds for $B$, with $\ell_k$ replaced by $\tau_k\ell_k$, and $u_j$ replaced by $\nu_j$. The proof is identical, and is also presented in Sec.\ \ref{spec_pf}.

\subsection{Optimal singular value shrinkage}
\label{sec:svshrink}

Theorem \ref{spec} describes precisely the limiting spectral theory of the matrix $\tilde{B} / \sqrt{n}$. Specifically, we derived formulas for the limiting cosines $c_k$ and $\tilde{c}_k$ of the angles between the top $r$ singular vectors of $\tilde{B} / \sqrt{n}$ and $X / \sqrt{n}$, and the relationship between the top singular values of these matrices. 

It turns out, following the work of \cite{gavish-donoho-2017} and \cite{nadakuditi2014optshrink}, that this information is sufficient to derive the optimal singular value shrinkage predictor of $X$. It is shown in \cite{gavish-donoho-2017} that $\lambda_i^* = \ell_k^{1/2} c_k \tilde{c}_k$, under the convention $c_k,\tilde{c}_k > 0$. Furthermore, the AMSE of this predictor is given by $\sum_{k=1}^r \ell_k(1 - c_k^2 \tilde{c}_k^2)$. We outline the derivation of these formulas in Sec.\ \ref{sv-derived}, though the reader may wish to refer to \cite{gavish-donoho-2017} for a more detailed description of the method, as well as extensions to other loss functions.

We next show how to derive consistent estimators of the angles and the limiting singular values of the observed matrix. Plugging these into the expression $\lambda_i^* = \ell_i^{1/2} c_i \tilde{c}_i$, we immediately obtain estimators of the optimal singular values $\lambda_i^*$. This will complete the proof that the algorithm given in Sec.\ \ref{in_sample_alg} solves the problem posed in Sec.\ \ref{Normalization} and defines the EBLP.

\subsubsection{Estimating $\ell_k$, $c_k$ and $\tilde{c}_k$}
To evaluate the optimal $\lambda_i^*$, we estimate the values of $\ell_k$, $c_k$, and $\tilde{c}_k$ using Thm.\ \ref{spec} whenever $\ell_k \ge b_H^2$ (that is, if the signal is strong enough). From \eqref{sv_eq} we have the formula $\ell_k = 1 / D_{\gamma,H}(t_k^2)$ where $t_k$ is the limiting singular value of the observed matrix $\tilde{B} / \sqrt n$. We also have the formulas \eqref{cos_inn22} and \eqref{cos_out44} for $c_k$ and the $\tilde{c}_k$.

We will estimate the Stieltjes transform $m_{\gamma,H}(z)$ by the \textit{sample Stieltjes transform}, defined as:
\begin{equation*}
\label{st_plugin}
    \hat{m}_{\gamma,H}(z) = \frac{1}{p-r}\sum_{k=r+1}^p \frac{1}{\lambda_k - z},
\end{equation*}
where the sum is over the bottom $p-r$ eigenvalues $\lambda_k$ of $\tilde B^\top \tilde B / n$. It is shown by \cite{nadakuditi2014optshrink} that $\hat{m}_{\gamma,H}$ is a consistent estimator of $m_{\gamma,H}$, and that using the corresponding plug-in estimators of $\underline{m}_{\gamma,H}$, $D_{\gamma,H}$ and $D_{\gamma,H}^\prime$, we also obtain consistent estimators of $\ell_k, c_k$, and $\tilde{c}_k$.

\subsubsection{Using $\hat{M}$ in place of $M$}
\label{estim_M}

To make the procedure fully implementable, we must be able to estimate the mean matrix $M = \mathbb{E}A_i^\top A_i$. If $M$ is estimated from the $n$ iid matrices $A_i^\top A_i$ by the sample mean $\hat{M} = n^{-1} \sum_{i=1}^n A_i^\top A_i$, we show that multiplying by $\hat{M}^{-1}$ has asymptotically the same effect as multiplying by the true $M^{-1}$, assuming that the diagonal entries of $M$ are bounded below. This justifies our use of $\hat M$.

\begin{lem}
\label{hat_M}
Suppose that the entries $M_i$ of $M$ are bounded away from 0: $M_i \ge \delta$ for some $\delta>0$, for all $i$. Let $\hat{M} = n^{-1}\sum_{i=1}^n A_i^\top A_i$. Then 
$$
    \lim_{p,n\to\infty} n^{-1/2} \| B M^{-1} -  B\hat{M}^{-1} \|_{op} = 0.
$$
\end{lem}
See Sec.\ \ref{hat_M_pf} for the proof. Note that the condition of this lemma are violated only when the entries of $M$ can be arbitrarily small; but in this case, the information content in the data on the corresponding coordinates vanishes, so the problem itself is ill-conditioned. The condition is therefore reasonable in practice.


%
%
%
%
%
%
%
%
%
%
%
%
%
%
\subsection{Prediction for weighted loss functions: whitening and big data}
\label{W_norm}
In certain applications there may be some directions that are more important than others, whose accurate prediction is more heavily prized. We can capture this by considering weighted Frobenius loss functions $\|\hat{X}_i - X_i\|_W^2 = \|W(\hat{X}_i - X_i)\|^2$, where $W$ is a positive-definite matrix. Can we derive optimal shrinkers with respect to these weighted loss functions?

The weighted error can be written as $\|\hat{X}_i - X_i\|_W^2 = \|W(\hat{X}_i - X_i)\|^2 = \|\widehat{W X_i} - W X_i \|^2$. In other words, the problem of predicting $X_i$ in the $W$-norm is identical to predicting $WX_i$ in the usual Frobenius norm. Because the vectors $W X_i$ lie in an $r$-dimensional subspace (spanned by $W u_1,\dots, W u_r$), the same EBLP method we have derived for $X_i$ can be applied to prediction of $W X_i$, assuming that the technical conditions we imposed for the original model hold for this transformed model. That is, we perform singular value shrinkage on the matrix of transformed observations $W \tilde{B}_i$.

To explore this further, recall that after applying the matrix $M^{-1}$ to each vector $A_i^\top Y_i$, the data matrix behaves asymptotically like the matrix with columns $X_i + \tilde{\ep}_i$, for some noise vectors $\tilde{\ep}_i$ that are independent of the signal $X_i$. The observations $W M^{-1} A_i^\top Y_i$ are asymptotically equivalent to $W X_i + W\tilde{\ep}_i$. If we choose $W$ to be the square root of the inverse covariance of $\tilde{\ep}_i$, then the effective noise term $W \tilde{\ep}_i$ has a identity covariance; we call this transformation ``whitening the effective noise''.

One advantage of whitening is that there are closed formulas for the asymptotic spikes and cosines. This is because the Stieltjes transform of white noise has an explicit closed formula; see \cite{bai2009spectral}. To make sense of the formulas, we will assume that the low-rank model $W X_i$ satisfies the assumptions we initially imposed on $X_i$; that is, we will assume:
\begin{align}
\label{eq-wx}
    W X_i = \sum_{k=1}^r \tilde{\ell}_k^{1/2} \tilde{z}_{ik} \tilde{u}_k
\end{align}
where the $z_{ik}$ are iid and the $\tilde{u}_k$ are orthonormal. With this notation, the empirical eigenvalues of $W\tilde{B}^\top \tilde{B} W / n$ converge to
\begin{align*}
    \lambda_k = 
    \begin{cases}
    (\tilde{\ell}_k + 1)\left( 1 + \frac{\gamma}{\tilde{\ell}_k}\right)
        & \text{ if } \tilde{\ell}_k > \sqrt{\gamma}, \\
    (1 + \sqrt{\gamma})^{2} 
        & \text{ otherwise}
    \end{cases}
\end{align*}
while the limit of the cosine of the angle between the $k^{th}$ empirical PC $\hat{u}_k$ and the $k^{th}$ population PC $u_k$ is
\begin{align}
\label{cos_inner}
    c_k^2 = 
    \begin{cases}
    \frac{1 - \gamma/\tilde{\ell}_k^2}{1 + \gamma/\tilde{\ell}_k}
        & \text{ if } \tilde{\ell}_k > \sqrt{\gamma}, \\
    0 
        & \text{ otherwise}
    \end{cases}.
\end{align}
and the limit of the cosine of the angle between the $k^{th}$ empirical left singular vector $\hat{v}_k$ and the $k^{th}$ left population singular vector $v_k$ is
\begin{align}
\label{cos_outer}
    \tilde{c}_k^2 = 
    \begin{cases}
    \frac{1 - \gamma/\tilde{\ell}_k^2}{1 + 1/\tilde{\ell}_k}
        & \text{ if } \tilde{\ell}_k > \sqrt{\gamma}, \\
    0 
        & \text{ otherwise}
    \end{cases}.
\end{align}

These formulas are derived in \cite{benaych2012singular}; also see \cite{paul2007asymptotics}.

Following Sec.\ \ref{sec:svshrink},  the $W$-AMSE of the EBLP is $\sum_{k=1}^r \tilde{\ell}_k (1-c_k^2 \tilde{c}_k^2)$. Since the parameters $\tilde{\ell}_k$, $c_k$ and $\tilde{c}_k$ are estimable from the observations, the $W$-AMSE can be explicitly estimated.

Using these formulas makes evaluation of the optimal shrinkers faster, as we avoid estimating the Stieltjes transform from the bottom $p-r$ singular values of $\tilde B$. Using whitening, the entire method only requires computation of the top $r$ singular vectors and values. Whitening thus enables us to scale our methods to extremely large datasets.


%
%
%

\subsubsection{Estimating the whitening matrix $W$}
\label{sec-est-W}
In the observation model $Y_i = A_i^\top X_i + \ep_i$, if the original noise term $\ep_i$ has identity covariance, that is $\Sigma_\ep = I_p$, then it is straightforward to estimate the covariance of the ``effective'' noise vector $\tilde{\ep}_i = M^{-1} A_i^\top \ep_i$, and consequuently to estimate the whitening matrix $W = \Sigma_{\tilde{\ep}}^{-1/2}$.

It is easy to see that $A_i^\top \ep_i$ has covariance $M = \mathbb{E}[A_i^\top A_i]$, which is diagonal. Then the covariance of $\tilde{\ep}_i$ is $M^{-1} M M^{-1}  = M^{-1}$, and $W = M^{1/2}$. As in the proof of Lemma \ref{hat_M}, $W$ can be consistently estimated from the data by the sample mean $\sum_{i=1}^n (A_i^\top A_i)^{1/2} / n$.

\subsection{Selecting the rank}
\label{sec-rk}

Our method requires a preliminary rank estimate. Our results state roughly that, after backprojection, the linearly transformed spiked model becomes a spiked model. So we believe we may be able to adapt some popular methods for selecting the number of components in spiked models. There are many such methods, and it is not our goal to recommend a particular one. One popular method in applied work is a permutation method called parallel analysis \citep{buja:eyob:1992,dobriban2017factor}, for which we have proposed improvements \citep{dobriban2017deterministic}. For other methods, see \cite{kritchman2008determining,passemier2012determining}, and also \cite{yao2015large}, Ch.\ 11, for a review. 

If the method is strongly consistent, in the sense that the number of components is almost surely correctly  estimated, then it is easy to see that the entire proof works. Specifically, the optimal singular value shrinkers can be obtained using the same orthonormalization method, and they can also be estimated consistently. Thus, for instance the methods from \cite{passemier2012determining,dobriban2017deterministic} are applicable if the spike strengths are sufficiently large.

\section{Out-of-sample prediction}
\label{oos_pred_alg}

In Sec.\ \ref{deriv_opt}, we derived the EBLP for predicting $X_i$ from $Y_i = A_i X_i + \ep_i$, $i=1,\dots,n$. We found the optimal coefficients $\eta_k$ for the predictor $\sum_{k=1}^r \eta_k \langle \tilde B_i, \hat{u}_k\rangle \hat{u}_k$, where the $\hat{u}_k$ are the empirical PCs of the normalized back-projected data $\tilde B_i = \hat M^{-1} A_i^\top Y_i$.

Now suppose we are given another data point, call it $Y_0 = A_0 X_0 + \ep_0$, drawn from the same model, but independent of $Y_1,\dots,Y_n$, and we wish to predict $X_0$ from an expression of the form $\sum_{k=1}^r \eta_k \langle\tilde B_0, \hat{u}_k\rangle \hat{u}_k$. 

At first glance, this problem appears identical to the one already solved. However, there is a subtle difference: the new data point is \textit{independent of the empirical PCs} $\hat{u}_1,\dots,\hat{u}_r$. It turns out that this independence forces us to use a \textit{different} set of coefficients $\eta_k$ to achieve optimal prediction.

We call this the problem of \textit{out-of-sample prediction}, and the optimal predictor the \textit{out-of-sample EBLP}. To be clear, we will refer to the problem of predicting $Y_1,\dots,Y_n$ as \textit{in-sample prediction}, and the optimal predictor as the \textit{in-sample EBLP}. We call $(Y_1,A_1),\dots, (Y_n,A_n)$ the \textit{in-sample observations}, and $(Y_0,A_0)$ the \textit{out-of-sample observation}. 

One might object that solving the out-of-sample problem is unnecessary, since we can always convert the out-of-sample problem into the in-sample problem. We could enlarge the in-sample data to include $Y_0$, and let $\hat{u}_k$ be the empirical PCs of this extended data set. While this is true, it is often not practical for several reasons. First, in on-line settings where a stream of data must be processed in real-time, recomputing the empirical PCs for each new observation may not be feasible. Second, if $n$ is quite large, it may not be viable to store all of the in-sample data $Y_1,\dots, Y_n$; the $r$ vectors $\hat{u}_1,\dots,\hat{u}_r$ require an order of magnitude less storage.

In this section we will first present the steps of the out-of-sample EBLP. Then we will provide a rigorous derivation.  We will also show that the AMSEs for in-sample and out-of-sample EBLP with respect to squared $W$-norm loss are identical, where $W$ is the inverse square root of the effective noise covariance. This is a rather surprising result that gives statistical justification for the use of out-of-sample EBLP, in addition to the computational considerations already described.

\subsection{Out-of-sample EBLP}
\label{out_sample_alg}

The out-of-sample denoising method can be stated simply, similarly to the in-sample algorithm in Sec.\ \ref{in_sample_alg}. 
We present the steps below.

\begin{enumerate}

\item \emph{Input}:
The top $r$ in-sample empirical PCs $\hat{u}_1,\dots,\hat{u}_r$. Estimates of the eigenvalues $\hat\ell_1,\dots,\hat\ell_r$ and cosines $\hat c_1,\dots,\hat c_r$. An estimate $\hat \Sigma_{\tilde{\ep}}$ of the noise covariance $\Sigma_{\tilde{\ep}}$ of the normalized backprojected noise vectors $\tilde{\ep}_i = M^{-1} A_i^\top \ep_i$. The diagonal matrix $\hat M^{-1}$ which is the inverse of an estimate of the covariance matrix of the noise $\ep_i$, and an out-of-sample observation $(Y_0,A_0)$. 

\item 
Construct the vector $\tilde B_0 =\hat M^{-1} A_0^\top Y_0$.

\item
Compute estimators of the out-of-sample coefficients $\eta_1,\dots,\eta_r$. These are given by the formula $\hat \eta_k = \frac{\hat \ell_k \hat c_k^2}{\hat \ell_k \hat c_k^2 + \hat d_k}$, where $\hat d_k = \hat{u}_k^\top \hat \Sigma_{\tilde{\ep}} \hat{u}_k$.

\item \emph{Output}:
Return the vector $\hat{X}_0 = \sum_{k=1}^r \hat \eta_k \langle \tilde B_0 , \hat{u}_k \rangle \hat{u}_k$.

\end{enumerate}

\subsection{Deriving out-of-sample EBLP}

We now derive the out-of-sample EBLP described in Sec.\ \ref{out_sample_alg}. Due to the independence between the $(Y_0,A_0)$ and the empirical PCs $\hat{u}_k$, the derivation is much more straightforward than was the in-sample EBLP. Therefore, we present the entire calculation in the main body of the paper.

\subsubsection{Covariance of $M^{-1} A_i^\top Y_i$}
Let $\tilde B_i = M^{-1} A_i^\top Y_i = M^{-1} D_i X_i + M^{-1} A_i^\top \varepsilon_i$, with $X_i = \sum_{j=1}^r \ell_j^{1/2} z_{ij} u_j$ and $D_i = A_i^\top A_i$.  Let $R_i = X_i + M^{-1} A_i^\top \varepsilon_i = X_i + \tilde{\ep}_i$; so $\tilde B_i = R_i + E_i X_i$, with $E_i = I_p - M^{-1}A_i^\top A_i$.

Observe that
$$
    \text{Cov}(\tilde B_i) = \text{Cov}(R_i) + \text{Cov}(E_i X_i) 
        + \mathbb{E}R_i (E_i X_i)^\top + \mathbb{E} (E_i X_i)^\top R_i
$$
and also that
$$
    \mathbb{E}R_i (E_i X_i)^\top 
        = \mathbb{E} X_i X_i^\top E_i 
        + \mathbb{E} \tilde{\ep}_i X_i^\top E_i = 0
$$
since $\mathbb{E}E_i = 0$ and $\mathbb{E} \varepsilon_i = 0$, and they are independent of $X_i$; similarly $ \mathbb{E} (E_i X_i)^\top R_i = 0$ as well. Consequently,
$$
    \text{Cov}(\tilde B_i) = \text{Cov}(R_i) + \text{Cov}(E_i X_i).
$$


Let $c_j = \mathbb{E}E_{ij}^2$. Then
\begin{align*}
  \mathbb{E} (E_i X_i) (E_i X_i)^\top
    &= \sum_{j=1}^r \ell_j^{1/2}
    \left(
    \begin{array}{c c c c  }
    c_1 u_{j1}^2 &                  &      &                      \\
                      &c_2 u_{j2}^2 &      &                      \\
                      &                  & \ddots &               \\ 
                      &                  &       &   c_p u_{jp}^2 
    \end{array}
    \right) 
    \nonumber \\
\end{align*}
which goes to zero in operator norm as $n,p\to\infty$, by the incoherence property of the $u_k$'s, and because $c_j$ are uniformly bounded under the assumptions of Theorem \ref{spec}. Therefore $\| \Sigma_{\tilde  B} - (\Sigma_X + \Sigma_{\tilde{\ep}})\|_{op} \to 0$.

\subsubsection{Out-of-sample coefficients and AMSE}
We will compute the optimal (in sense of AMSE) coefficients for out-of-sample prediction. We have normalized, back-projected observations $\tilde B_i = M^{-1} D_i X_i + \tilde{\ep}_i$, with $X_i = \sum_{j=1}^r \ell_j^{1/2} z_{ij} u_j$ and $\tilde{\ep}_i = M^{-1} A_i^\top \ep_i$.

We are looking for the coefficients $\eta_1,\dots,\eta_r$ so that the estimator
\beq
\label{oos_pred}
    \hat{X}_0^\eta = \sum_{j=1}^r \eta_j \langle \tilde B_0, \hat{u}_j \rangle \hat{u}_j
\eeq
has minimal AMSE. Here, $\hat{u}_j$ are the empirical PCs based on the in-sample data $(Y_1,A_1),\dots,(Y_n,A_1)$ (that is, the top $r$ eigenvectors of $\sum_{j=1}^n \tilde  B_i \tilde  B_i^\top$), whereas $(Y_0,A_0)$ is an out-of-sample datapoint.

It is easily shown that the contribution of $\eta_k$ to the overall MSE is:
$$
    \ell_k + \eta_k^2 \mathbb{E}(\hat{u}_k^\top \tilde  B_0)^2
           - 2 \eta_k \ell_k^{1/2}
             \mathbb{E}z_{0k}(\hat{u}_k^\top \tilde  B_0)(\hat{u}_k^\top u_k).
$$
It is also easy to see that the interaction terms obtained when expanding the MSE vanish.

To evaluate the quadratic coefficient above, first take the expectation over $Y_0$ and $A_0$ only, which gives:
\begin{align*}
    \mathbb{E}_0 (\hat{u}_k^\top \tilde  B_0)^2 
        = \hat{u}_k^\top \Sigma_{\tilde  B} \hat{u}_k
    &\sim \hat{u}_k^\top  \left(\sum_{j=1}^r \ell_j u_j u_j^\top 
                           + \Sigma_{\tilde{\ep}}\right) 
        \hat{u}_k 
        \nonumber \\
    & \sim \ell_k c_k^2 
           + \hat{u}_k^\top \Sigma_{\tilde{\ep}} \hat{u}_k 
\end{align*}

Note that when the original noise $\ep_i$ is white (i.e.\ $\Sigma_\ep = I_p$), we can estimate $d_{k} \equiv \hat{u}_k^\top \Sigma_{\tilde{\ep}} \hat{u}_k$ using the approximation $\Sigma_{\tilde{\ep}} \sim M^{-1}$, as in Sec.\ \ref{sec-est-W}. Defining the estimator $\hat{d}_k = \hat{u}_k^\top M^{-1} \hat{u}_k$ (or $\hat{u}_k^\top \hat{M}^{-1} \hat{u}_k$, where $\hat{M} = \sum_{i=1}^n A_i^\top A_i / n$), we therefore have $|\hat d_k - d_k| \to 0$.

Now turn to the linear term. We have $\hat{u}_k^\top \tilde  B_0 = \sum_{j=1}^r \ell_j^{1/2} z_{0j} \hat{u}_k^\top M^{-1} D_0 u_j + \hat{u}_k^\top \varepsilon_0$; using $\mathbb{E}[M^{-1}D_0] = I_p$ and using the almost sure convergence results, it follows after some simple calculation that $\ell_k^{1/2}\mathbb{E}[z_{0k}\hat{u}_k^\top\tilde  B_0 \hat{u}_k^\top u_k] \to \ell_k c_k^2$. Consequently, the mean-squared error of the out-of-sample predictor (as a function of $\eta_k$) is asyptotically equivalent to:
$$
    \sum_{k=1}^r \left\{
           \ell_k + \eta_k^2 (\ell_k c_k^2 + d_k) - 2 \eta_k \ell_k c_k^2 
     \right\} .
$$
This is minimized at 
$  \eta_k^* = \frac{\ell_k c_k^2}{\ell_k c_k^2 + d_k}
$
and the MSE is asymptotically equivalent to:
$$
    \sum_{k=1}^r \left( \ell_k - \frac{\ell_k^2 c_k^4}{\ell_k c_k^2 + d_k} \right).
$$

This finishes the derivation of the optimal coefficients for out-of-sample prediction.

\subsection{The whitened model}
Following the approach described in Sec.\ \ref{W_norm}, we can optimally predict $X_0$ using the $W$-loss, for any positive semi-definite matrix $W$. This is equivalent to performing optimal prediction of the signal $W X_0$ based on the observations $W \tilde B_0 = W M^{-1} D_0 X_0 + W \tilde{\ep}_0$ in the usual Frobenius sense.

We can always transform the data so that the effective noise $W \tilde{\ep} = W M^{-1} A_0^\top \tilde{\ep}_0$ has identity covariance; that is, take $W = \Sigma_{\tilde{\ep}}^{-1/2}$.

In this setting, the parameters $\hat{u}_k^\top W \Sigma_{\tilde{\ep}}^{-1/2} W  \hat{u}_k = \hat{u}_k^\top \hat{u}_k = 1$, and so $d_k = 1$. Consequently, the limiting AMSE is
\begin{align}
\label{oos-amse-white}
    \sum_{k=1}^r \left( \tilde{\ell}_k 
                       - \frac{\tilde \ell_k^2 c_k^4}{\tilde{\ell}_k c_k^2 + 1} 
                 \right)
\end{align}
where $\tilde{\ell}_k$ are the eigenvalues of the whitened model $W X_i$, assuming the model \eqref{eq-wx}. Using the formulas \eqref{cos_inner} and \eqref{cos_outer} for $c_k$ and $\tilde{c}_k$ as functions of $\tilde{\ell}_k$, it is straightforward to check that formula \eqref{oos-amse-white} is equal to $\sum_{k=1}^r \tilde{\ell}_k (1 - c_k^2 \tilde{c}_k^2)$, which is the in-sample AMSE with $W$-loss; we will show this in Sec.\ \ref{oos-proof}.  That is, the AMSE for whitened observations are identical for in-sample and out-of-sample EBLP. 

Thus, we state the following theorem:

\begin{theorem}[Out-of-sample EBLP]
\label{oos-prop}
Suppose our observations have the form $Y_i = A_iX_i + \ep_i$, $i=1,\ldots,n$, under the conditions of Thm.\ \ref{spec}, and suppose in addition that \eqref{eq-wx} holds, with $W = \Sigma_{\tilde \ep}^{-1/2}$ and $\tilde{\ep}_i = M^{-1}A_i^\top \ep_i$.

Given an out-of-sample observation $Y_0,A_0$, consider a predictor of $X_0$ of the form \eqref{oos_pred}. Then, for the optimal choice of $\eta_k$, the minimum asymptotic out-of-sample MSE achieved by this predictor in $\Sigma_{\tilde{\ep}}^{-1/2}$-norm equals the corresponding expression for in-sample MSE.

Thus, asymptotically, out-of-sample denoising is not harder than in-sample denoising.
\end{theorem}

The remainder of the proof of Thm.\ \ref{oos-prop} is contained in Sec.\ \ref{oos-proof}.

\section{Matrix denoising and missing data}
\label{mx_comp}
A well-studied problem to which our analysis applies is the problem of missing data, where coordinates are discarded from the observed vectors. Here the operators $D_i = A_i^\top A_i$ place zeros in the unobserved entries.

Without additive noise, recovering the matrix $X = [X_1,\dots,X_n]^\top$ is known as \emph{matrix completion}, and has been widely studied in statistics and signal processing. There are many methods with guarantees of exact recovery for certain classes of signals \citep{candes2009exact,candes2010power,jain2013,keshavan2010matrix,recht2011simpler, jain2013}.


Many methods for matrix completion assume that the target matrix $X$ is low-rank. This is the case for the linearly-transformed model as well, since the rows $X_i^\top$ of $X$ all lie in the $r$-dimensional subspace spanned by $u_1,\dots,u_r$. In the linearly-transformed model, the low-rank target matrix $X$ is itself random, and the analysis we provide for the performance of EBLP is dependent on this random structure.

Our approach differs from most existing methods. Our methods have the following advantages:

\begin{enumerate}
\item {\bf Speed.} Typical methods for matrix completion are based on solving optimization problems such as nuclear norm minimization \citep{candes2009exact, candes2010power}. These require iterative algorithms, where an SVD is computed at each step. In contrast, when an upper bound on the rank of the target matrix is known a priori our methods require only one SVD, and are thus much faster. Some of the methods for rank estimation in the spiked model discussed in Sec. \ref{sec-rk}, such as \cite{dobriban2017deterministic} and \cite{kritchman2008determining}, require only one SVD as well; we believe that these methods can be adapted to the linearly-transformed spiked model, though this is outside the scope of the current paper.

\item {\bf Robustness to high levels of noise.} Most matrix completion methods have guarantees of numerical stability: when the observed entries are accurate to a certain precision, the output will be accurate to almost the same precision. However, when the noise level swamps the signal, these stability guarantees are not informative. While many matrix completion methods can be made more robust by incorporating noise regularization, EBLP is designed to directly handle the high-noise regime. In Sec.\ \ref{simu}, we show that our method is more robust to noise than regularized nuclear norm minimization.


\item {\bf Applicability to uneven sampling.} While many matrix completion methods assume that the entries are observed with equal probability, other methods allow for uneven sampling across the rows and columns. Our method of EBLP allows for a different probability in each column of $X$. In Sec.\ \ref{sec-uneven} we compare our method to competing methods when the column sampling probabilities exhibit varying degrees of non-uniformity. In particular, we compare to the OptShrink method for noisy matrix completion \citep{nadakuditi2014optshrink}, which is nearly identical to EBLP when the sampling is uniform, but is not designed for uneven sampling. We also compare to \emph{weighted} nuclear norm minimization, designed to handle the uneven sampling.


\item {\bf Precise performance guarantees.} Our shrinkage methods have precise asymptotic performance guarantees for their mean squared error.  The errors can be estimated from the observations themselves.
\end{enumerate}

In addition to these advantages, our method has the seeming shortcoming that unlike many algorithms for matrix completion, it never yields exact recovery. However, our methods lead to \emph{consistent} estimators in the low-noise regime. In our model low noise corresponds to large spikes $\ell$. It is easy to see that taking $\ell \to\infty$ we obtain an asymptotic MSE of $\E\|X_i - \hat X_i\|^2 = O(1)$, whereas $\E\|X_i\|^2 = \ell$. Thus the correlation $\text{corr}(\hat X_i,X_i)\to 1$ in probability, and we get consistent estimators. Thus we still have good performance in low noise.



%
%
%

\subsection{Simulations}
\label{simu}


In this section, we illustrate the finite-sample properties of our proposed EBLP with noise whitening. We compare this method to three other methods found in the literature. First is the OptSpace method of \cite{keshavan2010matrix}. This algorithm is designed for uniform sampling of the matrix and relatively low noise levels, although a regularized version for larger noise has been proposed as well \citep{keshavan2010regularization}. As we will see, OptSpace (without regularization) typically performs well in the low-noise regime, but breaks down when the noise is too high. We use the MATLAB code provided by Sewoong Oh on his website \url{http://swoh.web.engr.illinois.edu/software/optspace/code.html}. We note that, like EBLP, OptSpace makes use of a user-provided rank.

The second method is nuclear norm-regularized least squares (NNRLS), as described in \cite{candes2010noise}. In the case of uniform sampling, we minimize the loss function $\frac{1}{2}\| X_\Omega - Y_\Omega\|^2 + w \cdot \| X\|_*$, where $\|\cdot\|_*$ denotes the nuclear norm and $X_\Omega$ denotes the vector of $X$'s values on the set of observed entries $\Omega$. Following the recommendation in \cite{candes2010noise} we take $w$ to be the operator norm of the pure subsampled noise term; that is, $w = \| E_\Omega \|$, where $E$ is the matrix of noise. With this choice of parameter, when the input data is indistinguishable from pure noise the estimator returned is the zero matrix. When the noise is white noise with variance $\sigma^2$, then $w = \sigma (\sqrt{p} + \sqrt{n}) \sqrt{|\Omega| / (pn)}$ at noise variance $\sigma^2$. If the noise is colored, we determine $w$ by simulation; we note that the Spectrode method of \cite{dobriban2015efficient} might offer an alternative means of determining $w$. To solve the minimization, we use the accelerated gradient method of \cite{ji2009agm}.

When the sampling probabilities differ across the columns of $X$, we compare to a weighted nuclear norm minimization. This minimizes the loss function $\frac{1}{2}\| X_\Omega - Y_\Omega\|^2 + w \cdot \|  X C_i\|_*$, where $C$ is the diagonal matrix with entries $C_{ii} = \sqrt{p}_i$, and $p_i$ is the probability that column $i$ is sampled. Again, we choose $w$ so that if there is no signal (i.e.\ $X=0$), then the zero matrix is returned. This method has been widely studied \citep{srebro2010collaborative,negahban2011estimation,klopp2014noisy,chen2015completing}.

The third method is OptShrink \citep{nadakuditi2014optshrink}. OptShrink assumes the sampling of the matrix is uniform; when this is the case, the method is essentially identical to EBLP without whitening. However, for non-uniform sampling we find the EBLP outperforms OptShrink, especially as the noise level increases. In Sec.\ \ref{sec-colored}, we also compare EBLP with whitening to OptShrink (which does not perform whitening) with colored noise; we find that whitening improves performance as the overall noise level increases. When using EBLP and OptShrink with data that is not mean zero, we estimate the mean using the available-case estimator, and subtract it before shrinkage.

In Sec.\ \ref{sec-in-oos}, we compare in-sample and out-of-sample EBLP. We demonstrate a very good agreement between the RMSEs, as predicted by Thm.\ \ref{oos-prop}, especially at high sampling rates.


In Secs.\ \ref{sec-sparsity}, \ref{sec-uneven} and \ref{sec-colored}, we used the following experimental protocol. The signals $X_i$ are drawn from a rank 10 model, with eigenvalues $1,2,\dots,10$, and random mean. Except for Sec.\ \ref{sec-sparsity}, the PCs $u_1,\dots,u_{10}$ were chosen to span a completely random 10-dimensional subspace of $\mathbb{R}^{300}$. We used the aspect ratio $\gamma = 0.8$, corresponding to a sample size of $n = 375$. The random variables $z_{ik}$ were taken to be Gaussian, as was the additive noise. The matrices $A_i$ are random coordinate selection operators, with each coordinate chosen with a given probability. When each entry of the matrix has probability $\delta$ of being selected, we will call $\delta$ the \emph{sampling rate}.

We measure the accuracy of a predictor $\hat{X}$ of the matrix $X$ using the root mean squared error, defined by
\begin{math}
    \|\hat{X} - X\|_F / \|X\|_F.
\end{math}
For each experiment, we plot the RMSEs of the different algorithms for forty runs of the experiment at increasing noise levels $\sigma$. The code for these experiments, as well as a suite of MATLAB codes for singular value shrinkage and EBLP, can be found online at \url{https://github.com/wleeb/opt-pred}.

\subsubsection{Sparsity of the PCs}
\label{sec-sparsity}

\begin{figure}[h]
\centering
  \includegraphics[scale=.33]{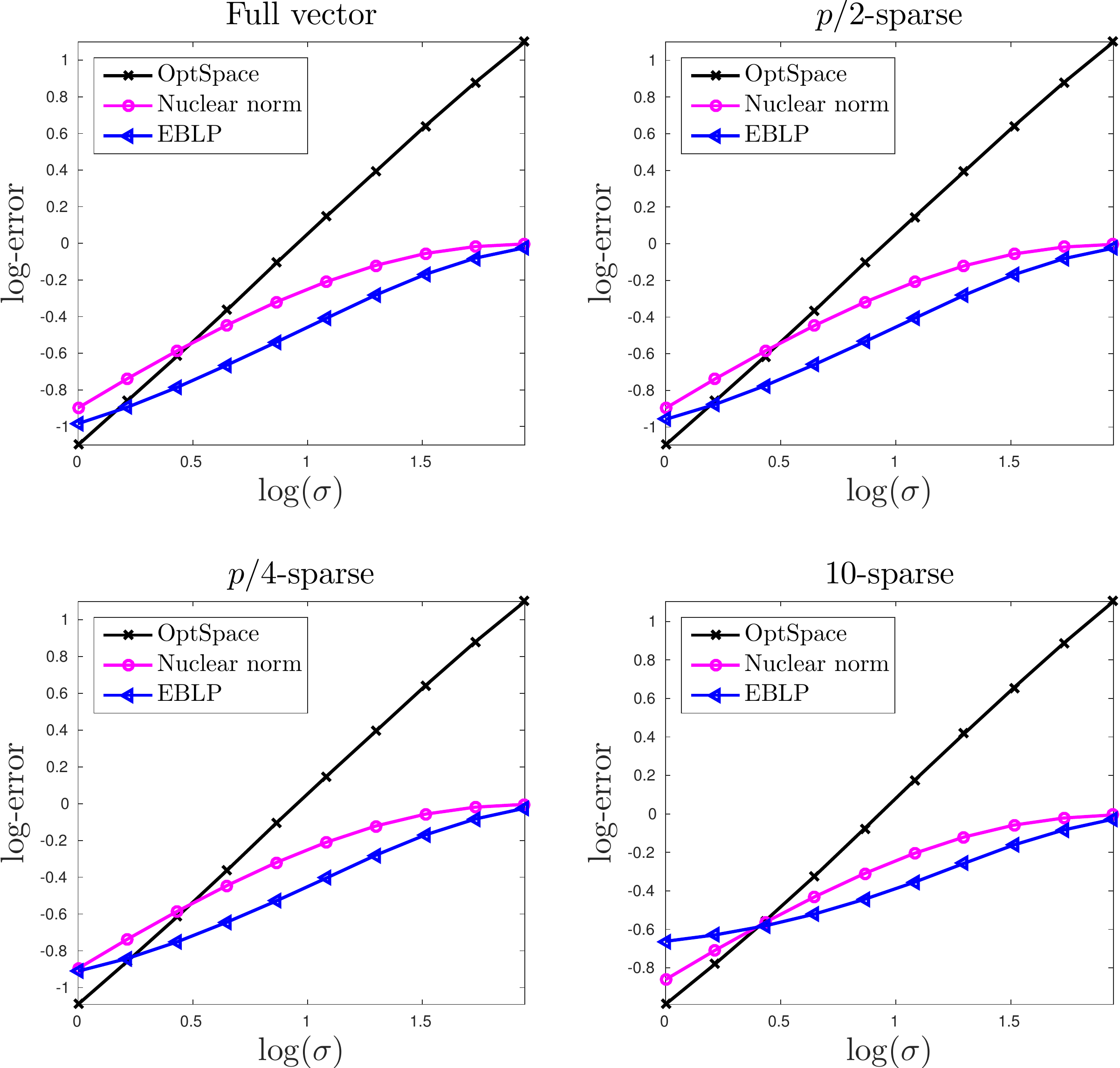}
\caption{
Log-RMSEs against log-noise for matrix completion. Each plot shows a different amount of sparsity in the PCs $u_1,\dots,u_{10}$.
}
\label{fig-sparsity}
\end{figure}

We compare the matrix completion algorithms when the PCs $u_1,\dots,u_{10}$ have different amounts of sparsity. We say that a vector is $m$-sparse if only $m$ coordinates are non-zero; we consider the cases where all the PCs are 10-sparse, $p/4$-sparse, $p/2$-sparse, and dense. We show the results in Fig.\ \ref{fig-sparsity}. Note that EBLP outperforms OptSpace and NNRLS at high noise levels, while it does worse than OptSpace at low noise levels in all sparsity regimes, and worse than both competing methods at low noise levels when the PCs are sparse.

\subsubsection{Uneven sampling}
\label{sec-uneven}

In this experiment, each coordinate is assigned a different probability of being selected, where the probabilities range linearly from $\delta$ to $1-\delta$ for $\delta \in (0,1)$. In addition to NNRLS and OptSpace, we also compare EBLP to OptShrink \citep{nadakuditi2014optshrink}, which assumes uniform sampling. With uniform sampling, the two procedures are nearly identical. However, EBLP performs better when the sampling is non-uniform.

\begin{figure}[h]
\centering
  \includegraphics[scale=.33]{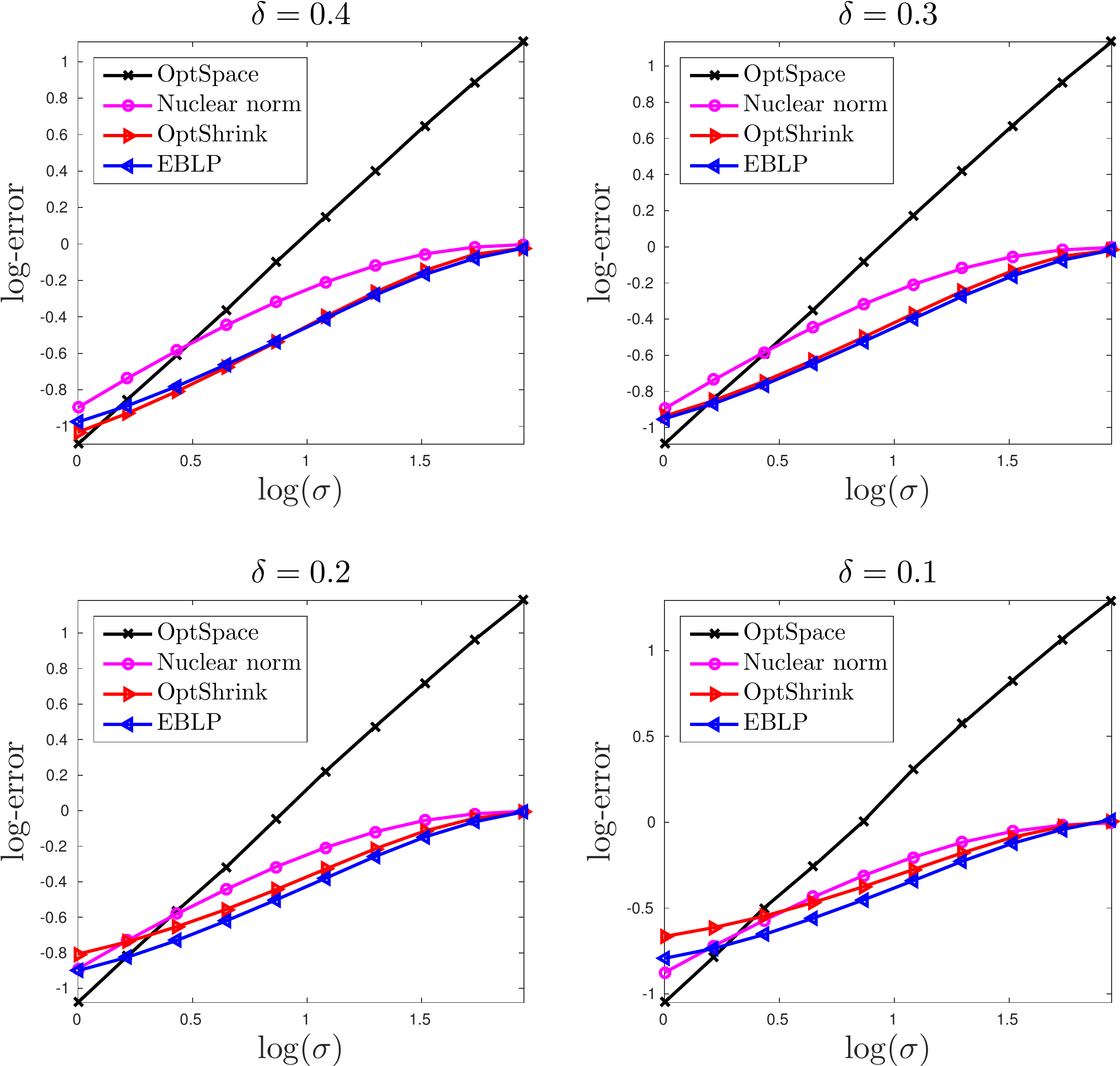}
\caption{
Log-RMSEs against log-noise for matrix completion. Each plot shows a different unevenness of sampling across the coordinates, with sampling probabilities ranging linearly from $\delta$ to $1-\delta$.
}
\label{fig-sparsity}
\end{figure}

\subsubsection{Colored noise}
\label{sec-colored}

We use colored noise whose covariance has condition number $\kappa > 1$. The noise covariance's eigenvalues increase linearly with the coordinates while having overall norm $p = 300$. In each experiment the noise is then multiplied by $\sigma$ to increase the overall variance of the noise while maintaining the condition number. We subsample uniformly with probability 0.5. Again, we compare EBLP with whitening to NNRLS, OptSpace, and OptShrink (which does not whiten). We observe that at high noise levels, EBLP with whitening outperforms OptShrink, while OptShrink performs better at low noise levels; and this effect increases with larger $\kappa$.

\begin{figure}[h]
\centering
  \includegraphics[scale=0.33]{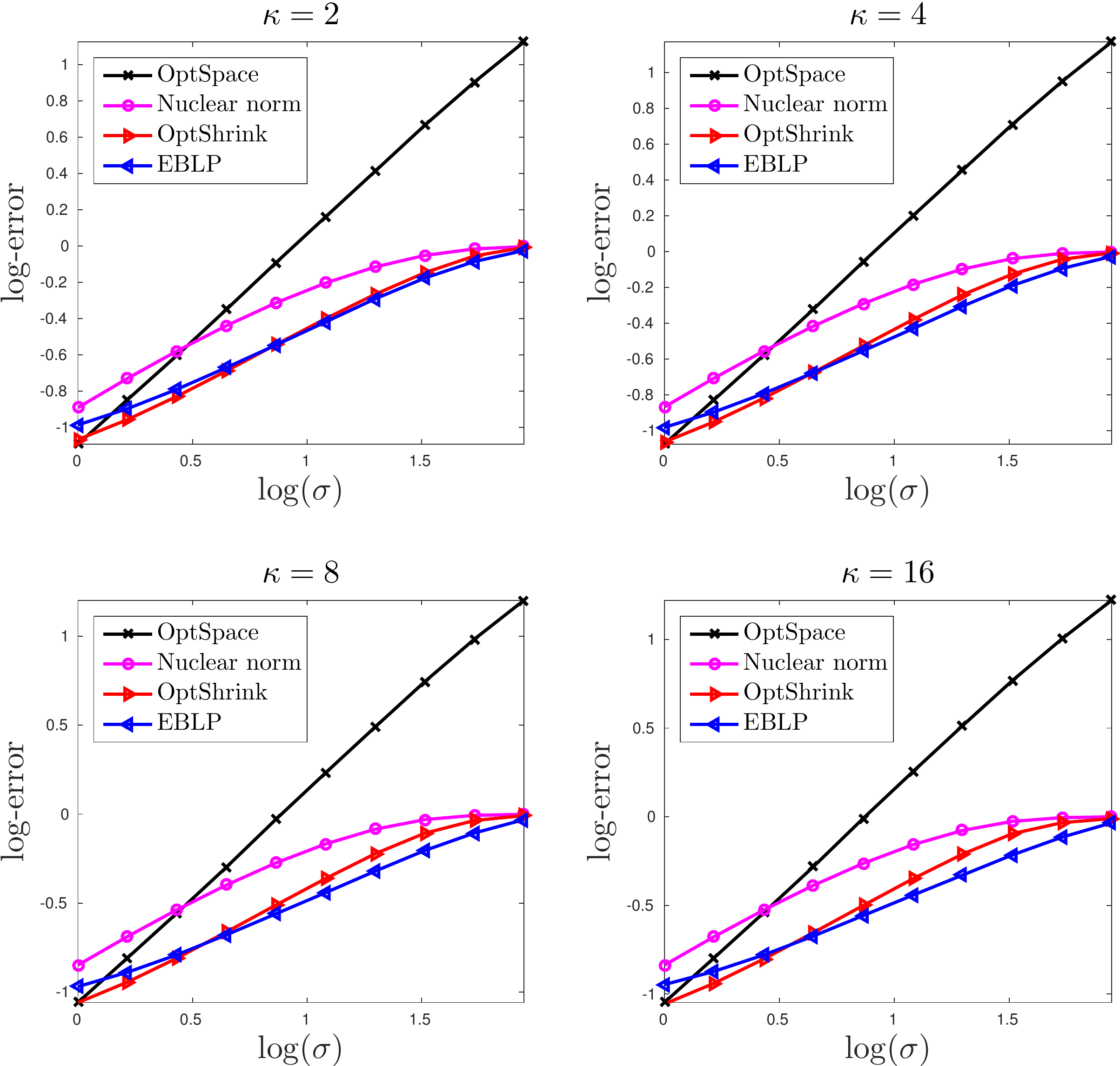}
\caption{
Log-RMSEs against log-noise for matrix completion. Each plot shows a different condition number $\kappa$ of the noise covariance matrix, reflecting different amounts of heterogeneity in the noise.
}
\label{fig-uneven}
\end{figure}

\subsubsection{In-sample vs.\ out-of-sample EBLP}
\label{sec-in-oos}

In this experiment, we compare the performance of in-sample and out-of-sample EBLP. Thm.\ \ref{oos-prop} predicts that asymptotically, the MSE of the two methods are identical. We illustrate this result in the finite-sample setting.

We fixed a dimension $p = 500$ and sampling rate $\delta$, and generated random values of $n > p$ and $\ell > 0$. For each set of values, we randomly generated two rank 1 signal matrices of size $n$-by-$p$, $X_{in}$ and $X_{out}$, added Gaussian noise, and subsampled these matrices uniformly at rate $\delta$ to obtain the backprojected observations $\tilde{B}_{in}$ and $\tilde{B}_{out}$. We apply the in-sample EBLP on $\tilde{B}_{in}$ to obtain $\hat{X}_{in}$, and using the singular vectors of $\tilde{B}_{in}$, we apply the out-of-sample EBLP to $\tilde{B}_{out}$ to obtain $\hat{X}_{out}$.

In Fig.\ \ref{fig-oos}, we show scatterplots of the RMSEs for the in-sample and out-of-sample data for each value of $n$ and $\ell$. We also plot the line $x=y$ for reference. The errors of in-sample and out-of-sample EBLP are very close to each other, though the finite sample effects are more prominent for small $\delta$.

\begin{figure}[h]
\centering
  \includegraphics[scale=0.33]{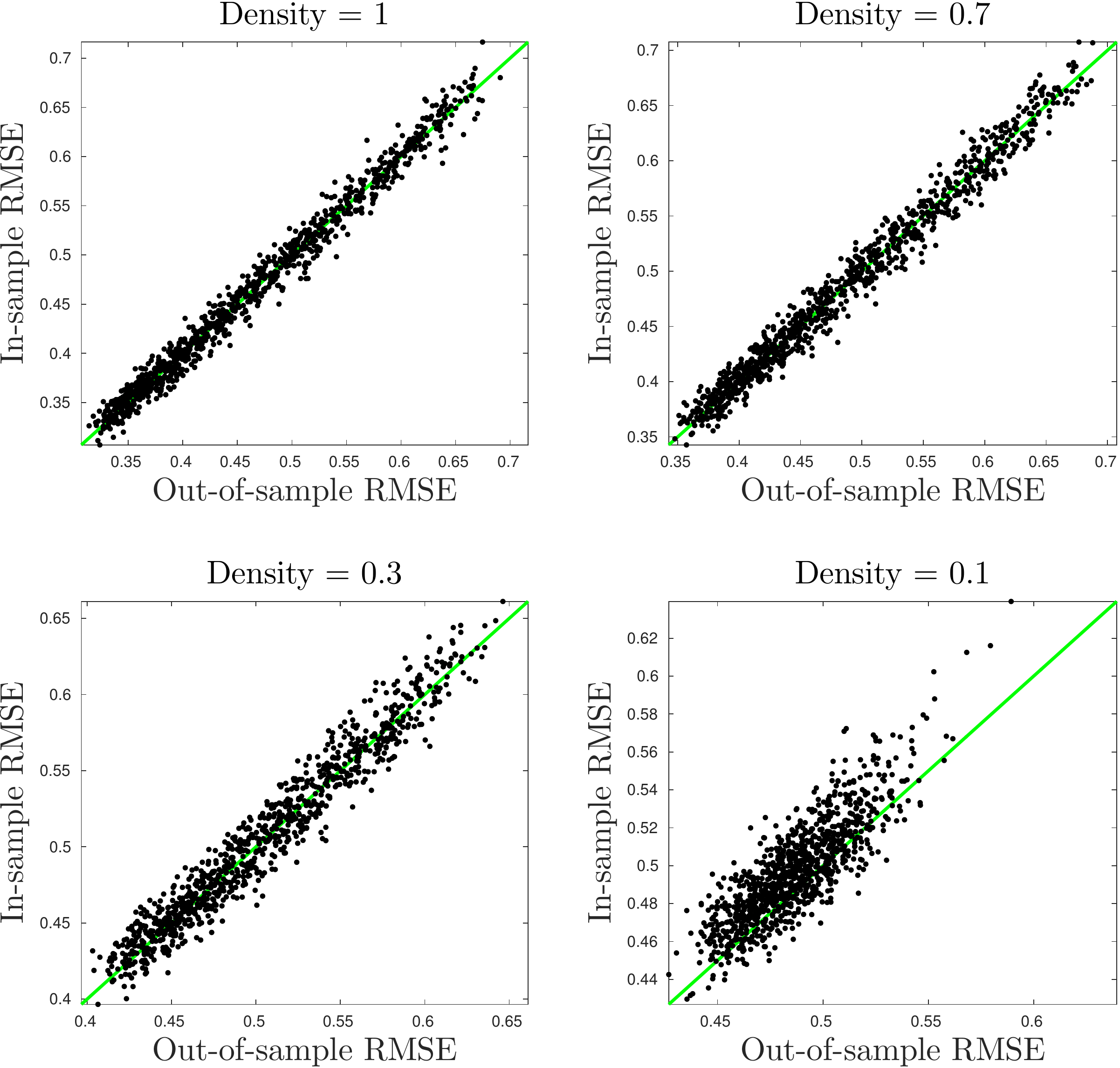}
\caption{
Scatterplots of the RMSEs of in-sample EBLP against out-of-sample EBLP for different sampling densities.
}
\label{fig-oos}
\end{figure}

\section{Conclusion}

In this paper we considered the linearly transformed spiked model, and developed asymptotically optimal EBLP methods for predicting the unobserved signals in the commutative case of the model, under high-dimensional asympotics. For missing data, we showed in simulations that our methods are faster, more robust to noise and to unequal sampling than well-known matrix completion methods.

There are many exciting opportunities for future research. One problem is to extend our methods beyond the commutative case. This is challenging because the asymptotic spectrum of the backprojected matrix $B$ becomes harder to characterize, and new proof methods are needed. Another problem is to understand the possible benefits of whitening. We saw that whitening enables fast optimal shrinkage, but understanding when it leads to improved denoising remains an open problem.

\section*{Acknowledgements}
The authors thank Joakim And\'{e}n, Tejal Bhamre, Xiuyuan Cheng, David Donoho, and Iain Johnstone for helpful discussions on this work. The authors are grateful to Matan Gavish for valuable suggestions on an earlier version of the manuscript. Edgar Dobriban was supported in part by NSF grant DMS-1407813, and by an HHMI International Student Research Fellowship. William Leeb was supported by the Simons Collaboration on Algorithms and Geometry. Amit Singer was partially supported by Award Number R01GM090200 from the NIGMS, FA9550-17-1-0291 from AFOSR, Simons Foundation Investigator Award and Simons Collaboration on Algorithms and Geometry, and the Moore Foundation Data-Driven Discovery Investigator Award.

%
%
%

\section{Proofs}

\subsection{Proof of Thm.\ \ref{spec}}
\label{spec_pf}

We present the proof of Thm.\ \ref{spec} for the backprojected matrix $B$. The proof for the normalized matrix $\tilde B$ is identical and omitted for brevity. The proof spans multiple sections, until Sec.\ \ref{pf_qf_lem3}.   To illustrate the idea, we first prove the single-spiked case, when $r=1$. The proof of the multispiked extension is provided in Sec.\ \ref{pf_multi}. 

Since $A_i^\top A_i = M+E_i$,  we have the following decomposition for the backprojected observations $B_i$:
$$
B_i =  A_i^\top Y_i  = M X_i + A_i^\top \varepsilon_i + E_i X_i.
$$
The first key observation is that after backprojection, we still have an \emph{approximate} spiked model. The new signal component is $MX_i$, the new noise is $ A_i^\top \varepsilon_i $. The error term $ E_i X_i$ has an asymptotically negligible contribution in operator norm, as shown below. The proof is provided in Sec.\ \ref{pf_op_norm_6}.

\begin{lem}
Let $E^*$ be the matrix formed by the vectors $n^{-1/2} E_i X_i$. Then the operator norm $\|E^*\|\to 0$ a.s.
\label{op_norm_6}
\end{lem}

Since the claims of our theorem concern the limits of the spectral distribution, singular values, and angles between population and sample singular vectors, all of which are continuous functions with respect to the operator norm, it follows that we can equivalently prove the theorem for 
$$
B_i^* = \ell^{1/2} z_i \cdot M u + A_i^\top  \ep_i.
$$

Let us denote by $\nu=M u/\xi^{1/2}$ the normalized backprojected signal, where $\xi  = |M u|^2 \to \tau$.  We will extend the technique of \cite{benaych2012singular} to characterize the spiked eigenvalues in this model. We denote the normalized vector $Z = n^{-1/2}\tilde Z$, with $\tilde Z = (z_1,\ldots,z_n)^\top$, the normalized noise $N = n^{-1/2}\mathcal{E}^*$, where $\mathcal{E}^* = (\ep_1^*,\ldots,\ep_n^*)^\top = (A_1^\top \ep_1,\ldots,A_n^\top \ep_n)^\top$ and the normalized backprojected data matrix $\tilde B^* = n^{-1/2}B^*$, where $B^* = (B_1^*,\ldots,B_n^*)^\top$. Then, our model is 
\beq
\label{sig_noise_2}
\tilde B^* = (\xi\ell)^{1/2}  \cdot Z \nu^\top + N. 
\eeq

 We will assume that $n,p \to \infty$ such that $p/n \to \gamma>0$. For simplicity of notation, we will first assume that $n \le p$, implying that $\gamma \ge 1$. It is easy to see that everything works when $n \ge p$.

By Lemma 4.1 of \cite{benaych2012singular}, the singular values of $ \tilde B^* $ that are not singular values of $N$ are the positive reals $t$ such that the 2-by-2 matrix
\begin{align*}
M_n(t) = 
\begin{bmatrix} 
t \cdot Z^\top (t^2 I_n - NN^\top)^{-1} Z & 
Z^\top (t^2 I_n - NN^\top)^{-1} N\nu \\ 
\nu^\top N^\top (t^2 I_n - NN^\top)^{-1} Z & 
t \cdot \nu^\top (t^2 I_p - N^\top N)^{-1} \nu
\end{bmatrix} 
-  
\begin{bmatrix} 0 & (\xi\ell)^{-1/2} \\  (\xi\ell)^{-1/2} & 0 \end{bmatrix}
\end{align*}

is not invertible, i.e., $\det[M_n(t)]=0$. We will find almost sure limits of the entries of $M_n(t)$, to show that it converges to a deterministic matrix $M(t)$. Solving the equation $\det[M(t)] = 0$ will provide an equation for the almost sure limit of the spiked singular values of $\tilde B^*$.  For this we will prove the following results: 

\begin{lem}[The noise matrix] 
\label{noise_lem}
The noise matrix $N$ has the following properties:
\benum
\item The eigenvalue distribution of $N^\top N$ converges almost surely (a.s.) to the Marchenko-Pastur distribution $F_{\gamma,H}$ with aspect ratio $\gamma \ge 1$. 
\item The top eigenvalue of $N^\top N$ converges a.s.\ to the upper edge $b_H^2$ of the support of $F_{\gamma,H}$. 
\eenum
\end{lem}

This is proved in Sec.\ \ref{pf_noise_lem}. For brevity we write $b  = b_H$.  Since $\tilde B^*$ is a rank-one perturbation of $N$, it follows that the eigenvalue distribution of $\tilde B^{*\top} \tilde B^*$ also converges to the MP law $F_{\gamma,H}$. This proves the first claim of Thm \ref{spec}. 

Moreover, since $NN^\top$ has the same $n$ eigenvalues as the nonzero eigenvalues of $N^\top N$,  the two facts in Lemma \ref{noise_lem} imply that when $t>b$, $n^{-1}\tr (t^2 I_n - NN^\top)^{-1} \to \int (t^2-x)^{-1} d\underline F_{\gamma,H}(x) = -\underline m(t^2)$. Here $\underline F_{\gamma,H}(x)  = \gamma F_{\gamma,H}(x)  + (1-\gamma)\delta_0$ and $\underline m = \underline m_{\gamma,H}$ is the Stieltjes transform of $\underline F_{\gamma,H}$. Clearly this convergence is uniform in $t$. As a special note, when $t$ is a singular value of the random matrix $N$, we formally define $(t^2 I_p - N^\top N)^{-1} =0$ and $(t^2 I_n - NN^\top)^{-1} =0$. When $t>b$, the complement of this event happens a.s. In fact, from Lemma \ref{noise_lem} it follows that $(t^2 I_p - N^\top N)^{-1}$ has a.s. bounded operator norm. Next we control the quadratic forms in the matrix $M_n$.

\begin{lem}[The quadratic forms] 
\label{qf_lem}
When $t>b$, the quadratic forms in the matrix $M_n(t)$ have the following properties:
\benum
\item $Z^\top (t^2 I_n - NN^\top)^{-1} Z - n^{-1}\tr (t^2 I_n - NN^\top)^{-1}\to 0$ a.s.
\item $Z^\top (t^2 I_n - NN^\top)^{-1} N\nu \to 0$ a.s.
\item $\nu^\top (t^2 I_p - N^\top N)^{-1} \nu \to -m(t^2)$ a.s., where $m = m_{\gamma,H}$ is the Stieltjes transform of the Marchenko-Pastur distribution $F_{\gamma,H}$. 
\eenum

Moreover the convergence of all three terms is uniform in $t>b+c$, for any $c>0$.
\end{lem}

This is proved in Sec.\ \ref{pf_qf_lem}.  The key technical innovation is the proof of the third part. Most results for controlling quadratic forms $x^\top A x$ are concentration bounds for random $x$. Here $x=\nu$ is fixed, and the matrix $A=(t^2 I_p - N^\top N)^{-1}$ is random instead. For this reason we adopt the ``deterministic equivalents'' technique of  \cite{bai2007asymptotics} for quantities $x^\top (z I_p - N^\top N)^{-1} x$, with the key novelty that we can take the imaginary part of the complex argument to zero. The latter observation is nontrivial, and mirrors similar techniques used recently in universality proofs in random matrix theory \citep[see e.g., the review by][]{erdos2012universality}.

Lemmas \ref{noise_lem} and \ref{qf_lem} will imply that for $t>b$, the limit of $M_n(t)$ is 
$$M(t) = 
\begin{bmatrix} 
-t \cdot \underline m(t^2) & 
-(\tau\ell)^{-1/2} \\ 
-(\tau\ell)^{-1/2} &
-t \cdot m(t^2) 
\end{bmatrix}.
$$

By the Weyl inequality, $\sigma_2(\tilde B^*) \le \sigma_2( (\xi\ell)^{1/2}  \cdot Z \nu^\top) + \sigma_1(N) = \sigma_1(N)$. Since $\sigma_1(N) \to b$ a.s. by Lemma \ref{noise_lem}, we obtain that  $\sigma_2(\tilde B^*) \to b$ a.s.  Therefore for any $\ep>0$, a.s. only $\sigma_1(\tilde B^*)$ can be a singular value of $\tilde B^*$ in $(b+\ep,\infty)$ that is not a singular value of $N$.

It is easy to check that $D(x) = x\cdot\underline m(x) m(x)$ is strictly decreasing on $(b^2,\infty)$.  Hence, denoting $h = \lim_{t\downarrow b} 1/D(t^2)$, for $\tau\ell>h$, the equation $D(t^2)=1/(\tau \ell)$ has a unique solution $t\in (b,\infty)$. By Lemma A.1 of \cite{benaych2012singular}, we conclude that for $\tau\ell>h$, $\sigma_1(\tilde B^*) \to t$ a.s., where $t$ solves the equation $\det[M(t)]=0$, or equivalently, 
$$
t^2 \cdot \underline m(t^2) m(t^2)  = \frac1{\tau\ell}.
$$
If $\tau\ell \le h$, then we note that $\det[M_n(t)]\to\det[M(t)]$ uniformly on $t>b+\ep$. Therefore, if $\det[M_n(t)]$ had a root $\sigma_1(\tilde B^*)$ in $(b+\ep,\infty)$, $\det[M(t)]$ would also need to have a root there, which is a contradiction. Therefore, we conclude $\sigma_1(\tilde B^*) \le b+\ep$ a.s., for any $\ep>0$. Since $\sigma_1(\tilde B^*) \ge \sigma_2(\tilde B^*) \to b$, we conclude that $\sigma_1(\tilde B^*) \to b$ a.s., as desired. This finishes the spike convergence claim in Thm.\ \ref{spec}.

Next, we turn to proving the convergence of the angles between the population and sample eigenvectors. Let $\hat Z$ and $\hat u$ be the singular vectors associated with the top singular value $\sigma_1(\tilde B^*)$ of $\tilde B^*$. Then, by Lemma 5.1 of \cite{benaych2012singular}, if $\sigma_1(\tilde B^*)$ is not a singular value of $X$, then the vector $\eta = (\eta_1,\eta_2)= (\nu^\top \hat u, Z^\top \hat Z)$ belongs to the kernel of the matrix $M_n(\sigma_1(\tilde B^*))$. By the above discussion, this 2-by-2 matrix is of course singular, so this provides one linear equation for the vector $\eta$ (with $R=(t^2 I_n - NN^\top)^{-1}$)
$$t \eta_1 \cdot Z^\top R Z 
+ \eta_2[ Z^\top R N \nu- (\xi \ell)^{-1/2}]=0.$$

By the same lemma cited above, it follows that we have the \emph{norm identity} (with $t = \sigma_1(\tilde B^*)$)
\beq
\label{norm_id}
t^2 \eta_1^2 \cdot Z^\top R^2 Z 
+ \eta_2^2 \cdot \nu^\top N^\top R^2 N \nu 
+ 2 t \eta_1 \eta_2  \cdot Z^\top R^2 N \nu 
= (\xi \ell)^{-1}.
\eeq

This follows from taking the norm of the equation $t \eta_1 \cdot R Z +\eta_2  \cdot R N \nu = (\xi \ell)^{-1/2} \hat Z$ (see Lemma 5.1 in \cite{benaych2012singular}). We will find the limits of the quadratic forms below.

\begin{lem}[More quadratic forms] 
\label{qf_lem2}
The quadratic forms in the norm identity have the following properties:
\benum
\item $Z^\top (t^2 I_n - NN^\top)^{-2} Z - n^{-1}\tr (t^2 I_n - NN^\top)^{-2}\to 0$ a.s.
\item $Z^\top (t^2 I_n - NN^\top)^{-2} N\nu \to 0$ a.s.
\item $\nu^\top N^\top (t^2 I_n - N N^\top)^{-2} N \nu \to  m(t^2) + t^2 m'(t^2)$ a.s., where $m$ is the Stieltjes transform of the Marchenko-Pastur distribution $ F_{\gamma,H}$. 
\eenum
\end{lem}

The proof is in Sec.\ \ref{pf_qf_lem2}. Again, the key novelty is the proof of the third claim. The standard concentration bounds do not apply, because $u$ is non-random. Instead, we use an argument from complex analysis constructing a sequence of functions $f_n(t)$ such that their derivatives are $f_n'(t) = \nu^\top N^\top (t^2 I_n - N N^\top)^{-2} N \nu$, and deducing the convergence of $f_n'(t)$ from that of $f_n(t)$. 

Lemma \ref{qf_lem2} implies that $n^{-1}\tr (t^2 I_n - NN^\top)^{-2} \to \int (t^2-x)^{-2} d\underline F_{\gamma,H}(x) = \underline m'(t^2)$ for $t>b$.  Solving for $\eta_1$ in terms of $\eta_2$ from the first equation, plugging in to the second, and taking the limit as $n \to \infty$, we obtain that $\eta_2^2 \to c_2$, where 

$$c_2  \left(\frac{\underline m'(t^2)}{\tau\ell\underline m(t^2)^2} + m(t^2)+t^2 m'(t^2)\right) = \frac{1}{\tau\ell}.$$

Using $D(x) = x\cdot m(x)\underline m(x)$, we find $c_2 =  \underline m(t^2)/[D'(t^2)\tau \ell]$, where
$t$ solves \eqref{sv_eq}.
From the first equation, we then obtain  $\eta_1^2 \to c_1$, where $c_1 =  m(t^2)/[D'(t^2)\tau \ell]$, where
$t$ is as above.
This finishes the proof of Thm.\ \ref{spec} in the single-spiked case. The proof of the multispiked case is a relatively simple extension of the previous argument, so we present it in Sec.\ \ref{pf_multi}.

\subsection{Proof of Lemma \ref{op_norm_6}}
\label{pf_op_norm_6}

Since $X_i = \ell^{1/2} z_i u$, 
the $(i,j)$-th entry of $E^*$ is  $n^{-1/2} \ell^{1/2} z_i \cdot E_{ij} u(j)$. Now, denoting by $\odot$ elementwise products
$$\|E^{*}\| = \sup_{\|a\|=\|c\|=1} a^\top E^{*} c = n^{-1/2} \ell^{1/2} \sup_{\|a\|=\|c\|=1}  (a \odot z)^\top E (c \odot u). $$
We have $\|a \odot z\| \le \|a\|  \max |z_i| = \max |z_i|$ and $\|c \odot u\| \le \|c\|  \max |u_i|=\max |u_i|$, hence
$$\|E^{*}\| \le  \ell^{1/2}\cdot \|n^{-1/2}E\| \max |z_i| \max |u_i|.$$

Below we will argue that $\|n^{-1/2}E\| \le C$, a.s., so that the the operator norm is a.s. bounded. (The constant $C$ can change from line to line.) Once we have established that, we have that, almost surely, 

$$\|E^{*}\| \le  C\| z\| _\infty \| u\| _\infty.$$

Then, our main claim follows as long as 

\beq\label{z}
\| z\| _\infty \| u\| _\infty \to_{a.s.} 0.
\eeq

This holds under several possible sets of assumptions: 
\benum

\item {\bf Polynomial moment assumption}. Suppose $\E|z_i|^m \le C$ for some $m>0$ and  $C < \infty$. Then, since $z$ has iid standardized entries, we can derive that 
$$Pr(\max |z_i|\ge a) \le \E \max |z_i|^{m}/a^{m} \le nC/a^{m}.$$
To ensure that these probabilities are summable, we take $m$ such that $n/a^{m} = C/n^{1+\phi'}$ for some small $\phi'>0$. This is equivalent to $a = Cn^{(2+\phi')/m}$. It follows that $Pr(\max |z_i|\le C n^{(2+\phi')/m})$ a.s., for any $\phi'>0$.

Therefore, the required delocalization condition on $u$ is that 
$$\| u\|_{\infty} \cdot n^{(2+\phi')/m} \to_{a.s} 0. $$
Therefore, it is enough to assume that $$\| u\|_{\infty} \le C n^{-(2/m+c')},$$
for any small constant $c'>0$. Since $\|u\| = 1$, we have that $\| u\|_{\infty}  \ge n^{-1/2}$, so this is only feasible for $m>4$. Since $n$ is proportional to $p$, we can replace $n$ by $p$ in our assumption.

\item {\bf Sub-gaussian assumption}. Suppose the $z_i$ are sub-gaussian in the sense that $\E \exp(t |z_i|^2) \le C$ for some $t>0$ and  $C < \infty$ \citep[e.g.,][]{vershynin2010introduction}. Then, since $z$ has iid standardized entries, we can derive that 
$$Pr(\max |z_i|\ge a) \le \E \max\exp( t|z_i|^{2})/\exp(t a^{2}) \le nC/\exp(t a^{2}) .$$
To ensure that these probabilities are summable, we take $m$ such that $n/\exp(t a^{2}) = C/n^{1+\phi'}$ for some small $\phi'>0$. This is equivalent to 

$$a = \sqrt{\frac{(2+\phi') \log n}{t}}.$$

Ignoring constants that do not depend on $n$, we can say that $a = C \sqrt{\log n}.$ Since $n$ is proportional to $p$, we can replace $n$ by $p$ in our assumption.

It follows that $Pr(\max |z_i|\le C  \sqrt{\log n})$ a.s., for some $C>0$.

Therefore, the required delocalization condition on $u$ is that 
$$\| u\|_{\infty} \cdot \sqrt{\log n} \to_{a.s} 0. $$
\eenum

It remains to show that $\|n^{-1/2}E\|$ is a.s. bounded. We will see below that for this it is enough that the $6+\ep$-th moment of all $E_{ij}$ is uniformly bounded for some $\ep>0$.  We will only give a proof sketch, as this consists of a small modification of Thm.\ 6.3 of \cite{bai2009spectral}. Their proof essentially goes through, except that one needs a slightly different, simpler, argument for the initial truncation step.  

For this, we follow the same steps as the original proof of Cor. 6.6 from \cite{bai2009spectral}. The first step (on p. 128 of \cite{bai2009spectral}) is a truncation, which relies on Thm.\ 5.8, a result controlling the operator norm of matrices with iid entries. The proof of this result is not provided in the book, since it is almost the same as Thm.\ 5.1. Therefore, we will show how to adapt the argument for Thm.\ 5.1 to our assumptions. 

The first step of that proof (on p. 94) is a truncation of the entries of $E$ at a threshold $cn^{1/2}$. Let $\tilde E_{ij} = E_{ij} I(|E_{ij}|\le cn^{1/2})$, and let $\tilde E_{ij}$ the corresponding matrix. We need to show that $P(E\neq \tilde E,\, i.o.)=0$. Let $A_n = \bigcup_{(i, j)\le (n,p)} I\left(|E_{ij}| \ge cn^{1/2}\right)$. Then $\{E\neq \tilde E\}=A_n$. By the Borel-Cantelli lemma, it is enough to show:

$$\sum_n P\left[A_n\right]
= \sum_n P\left[\bigcup_{(i, j)\le (n,p)} I\left(|E_{ij}| \ge cn^{1/2}\right)\right]<\infty.$$

We can bound 
$$ P[A_n]\le
 np \cdot \max_{i,j} P\left(|E_{ij}| \ge cn^{1/2}\right) 
 \le np \cdot \max_{i,j} \E |E_{ij}|^k/[c^kn^{k/2}].$$ 
 
By taking the exponent $k = 6+\ep$, we see from the assumption $\E E_{ij}^{6+\ep}<C$ that this bound is summable. Thus the first step of the proof of Thm.\ 5.1 adapts to our setting. Then, similarly to remark 5.7 on p.104, we obtain Thm.\ 5.8 under the present conditions. This shows that the first step of the proof of Cor. 6.6 goes through.
 
Continuing with the proof of Cor. 6.6 under the present conditions, we need on p. 129 that the conditions of Thm 5.9 be met. It is immediate to see that they are met. In the remaining part of the argument, as stated on p. 129, ``no assumptions [beyond independence and moment assumptions] need to be made on the relationship between the $E$-s for different $n$.''  Since our assumptions ensure that after truncation we are in the same setting as that stated in Thm.\ 6.3, including the independence and moment assumptions, the rest of the argument applies verbatim. This finishes the proof of the sufficiency of finite $6+\ep$ moment for $E$. This also finishes the argument that $\|n^{-1/2}E\|$ is a.s. bounded, and thus that of the main lemma.

\subsection{Proof of Lemma \ref{noise_lem}}
\label{pf_noise_lem}
 Recall that $N = n^{-1/2}\mathcal{E}^*$, where $\mathcal{E}^*$ has rows $\ep_i^* = A_i^\top  \ep_i$. According to our assumptions, the rows are iid, with independent entries having a distribution of variances $H_p$. Recall that we assumed that $H_p \Rightarrow H$. Hence the eigenvalue distribution of $N^\top N$ converges to the Marchenko-Pastur distribution $F_{\gamma,H}$. This follows essentially by Thm.\ 4.3 of \cite{bai2009spectral}. While that result is stated for identically distributed random variables, it is easy to see that under our higher moment assumptions this is not needed, as the truncation steps on p. 70 of \cite{bai2009spectral} go through; the situation is analogous to the modification of \citep[][Cor.\ 6.6]{bai2009spectral} from Lemma \ref{op_norm_6}. 

Moreover, since $\smash{\E{\ep_{ij}^{*(6+\ep)}}<\infty}$ and $\smash{\sup \supp(H_p) \to \sup \,\text{supp}(H)}$, the largest eigenvalue of $N^\top N$ converges a.s. to the upper edge $b^2$ of the support of $F_{\gamma,H}$, by the modification of \citep[][Cor.\ 6.6]{bai2009spectral} presented in the proof of Lemma \ref{op_norm_6}.

\subsection{Proof of Lemma \ref{qf_lem}}
\label{pf_qf_lem}

\textbf{Part 1}: 
For $Z^\top (t^2 I_n - NN^\top)^{-1} Z$, note that $Z$ has iid entries ---with mean 0 and variance $1/n$---that are independent of $N$. We will use the following result: 

\begin{lem}[Concentration of quadratic forms, consequence of Lemma B.26 in \citet{bai2009spectral}] Let $x \in \RR^k$ be a random vector with i.i.d. entries and $\EE{x} = 0$, for which $\EE{(\sqrt{k}x_i)^2} = 1$ and $\smash{\sup_i \EE{(\sqrt{k}x_i)^{4+\phi}}}$ $ < C$ for some $\phi>0$ and $C <\infty$. Moreover, let $A_k$ be a sequence of random $k \times k$ symmetric matrices independent of $x$, with a.s. uniformly bounded eigenvalues. Then the quadratic forms $x^\top A_k x $ concentrate around their means:  $\smash{x^\top A_k x - k^{-1} \tr A_k \rightarrow_{a.s.} 0}$.
\label{quad_form}
\end{lem}

 We apply this lemma with $x = Z$, $k=p$ and $A_p =  (t^2 I_n - NN^\top)^{-1}$. To get almost sure convergence, here it is required that $z_i$ have finite $4+\phi$-th moment. This shows the concentration of $Z^\top (t^2 I_n - NN^\top)^{-1} Z$. 

\textbf{Part 2}: 
To show $Z^\top (t^2 I_n - NN^\top)^{-1} N\nu$ concentrates around  0, we note that $w=(t^2 I_n - NN^\top)^{-1} N\nu$ is a random vector independent of $Z$, with a.s. bounded norm. Hence, conditional on $w$: 
\begin{align*}
Pr(|Z^\top w|\ge a|w) &\le  a^{-4} \E|Z^\top w|^4 = a^{-4} [\sum_{i} \E Z_{ni}^4w_i^4 +\sum_{i\neq j} \E Z_{ni}^2  \E Z_{nj}^2w_i^2w_j^2] \\
&\le a^{-4}  \E Z_{n1}^4 (\sum_i w_i^2)^2 = a^{-4} n^{-2} \E Z_{1}^4 \cdot \|w\|_2^4
\end{align*}

For any $C$ we can write
$$Pr(|Z^\top w|\ge a) \le Pr(|Z^\top w|\ge a|\|w\|\le C) +Pr(\|w\|> C).$$
For sufficiently large $C$, the second term, $Pr(\|w\|>C)$ is summable in $n$. By the above bound, the first term is summable for any $C$. Hence, by the Borel-Cantelli lemma, we obtain $|Z^\top w|\to 0$ a.s. This shows the required concentration.

\textbf{Part 2}: 
Finally we need to show that $\nu^\top (t^2 I_p - N^\top N)^{-1} \nu$ concentrates around  a definite value. This is probably the most interesting part, because the vector $\nu$ is not random. Most results for controlling expressions of the above type are designed for random $\nu$; however here the matrix $N$ is random instead. For this reason we will adopt a different approach.
 
 Under our assumption we have $\nu^\top(P -zI_p)^{-1} \nu \to m_H(z)$, for $z = t^2 + i v$ with $v>0$ fixed, where $P$ is the diagonal matrix with $(j,j)$-th entry $\Var{ \ep_{ij}^*}$. Therefore, Thm 1 of \cite{bai2007asymptotics} shows that $\nu^\top (z I_p - N^\top N)^{-1} \nu \to - m(z)$
a.s., where $m(z)$ is the Stieltjes transform of the Marchenko-Pastur distribution $F_{\gamma,H}$. 

A close examination of their proofs reveals that their result holds when $v \to 0$ sufficiently slowly, for instance $v = n^{-\alpha}$ for $\alpha = 1/10$.  The reason is that all bounds in the proof have the rate $n^{-k}v^{-l}$ for some small $k,l>0$, and hence they converge to 0 for $v$ of the above form. 

 For instance, the very first bounds in the proof of Thm 1 of \cite{bai2007asymptotics} are in Eq. (2.2) on page 1543. The first one states a bound of order $O(1/n^r)$. The inequalities leading up to it show that the bound is in fact  $O(1/(n^r v^{2r}))$. Similarly, the second inequality, stated with a bound of order $O(1/n^{r/2})$ is in fact  $O(1/(n^{r/2} v^{r}))$. These bounds go to zero when $v = n^{-\alpha}$ with small $\alpha>0$. In a similar way, the remaining bounds in the theorem have the same property.

To get the convergence for real $t^2$ from the convergence for complex $z = t^2 + iv$, we note that 
\begin{align*}
|\nu^\top (z I_p - N^\top N)^{-1} \nu - \nu^\top (t^2 I_p - N^\top N)^{-1} \nu| & 
= v |\nu^\top (z I_p - N^\top N)^{-1} (t^2 I_p - N^\top N)^{-1} \nu| \le \\
&\le v \|(t^2 I_p - N^\top N)^{-1}\|^2 \cdot u^\top u.
\end{align*}
As discussed above, when $t>b$, the matrices  $(t^2 I_p - N^\top N)^{-1}$ have a.s. bounded operator norm. Hence, we conclude that if $v\to0$, then a.s.
$$\nu^\top (z I_p - N^\top N)^{-1} \nu - \nu^\top (t^2 I_p - N^\top N)^{-1} \nu \to0.$$

 Finally, $m(z) \to m(t^2)$ by the continuity of the Stieltjes transform for all $t^2>0$ \citep{bai2009spectral}. We conclude that $\nu^\top (t^2 I_p - N^\top N)^{-1} \nu\to -m(t^2)$ a.s. This finishes the analysis of the last quadratic form. 

\subsection{Proof of Lemma \ref{qf_lem2}}
\label{pf_qf_lem2}
\textbf{Parts 1 and 2}: 
The proofs of Part 1 and 2 are exactly analogous to those in Lemma \ref{qf_lem}. Indeed, the same arguments work despite the change from $(t^2 I_p - N^\top N)^{-1}$ to $(t^2 I_p - N^\top N)^{-2}$, because the only properties we used are its independence from $Z$, and its a.s. bounded operator norm. These also hold for $(t^2 I_p - N^\top N)^{-2}$, so the same proof works. 

\textbf{Part 3}: 
We start with the identity $\nu^\top N^\top (t^2 I_n - N N^\top)^{-2} N \nu  = - \nu^\top (t^2 I_p - N^\top N)^{-1} \nu + t^2 \nu^\top (t^2 I_p - N^\top N)^{-2} \nu$. Since in Lemma \ref{qf_lem} we have already established $\nu^\top (t^2 I_p - N^\top N)^{-1} \nu \to -m(t^2)$, we only need to show the convergence of  $\nu^\top (t^2 I_p - N^\top N)^{-2} \nu$. 

For this we will employ the following \emph{derivative trick} \citep[see e.g.,][]{dobriban2015high}. We will construct a function with two properties: (1) its derivative is the quantity $\nu^\top (t^2 I_p - N^\top N)^{-2} \nu$ that we want, and (2) its limit is convenient to obtain. The following lemma will allow us to get our answer by interchanging the order of limits: 
\begin{lem}[see Lemma 2.14 in \cite{bai2009spectral}] Let $f_1, f_2,\ldots $ be analytic on a domain $D$ in the complex plane, satisfying $|f_n(z)| \le M$ for every $n$ and $z$ in $D$. Suppose that there is an analytic function $f$ on $D$ such that $f_n(z) \to f(z)$ for all $z \in D$. Then it also holds that $f_n'(z) \to f'(z)$ for all $z \in D$.
\label{holo_derivative_conv}
\end{lem}

Accordingly, consider the function $f_p(r) = - \nu^\top (r I_p - N^\top N)^{-1} \nu$. Its derivative is $f_p'(r) = \nu^\top (r I_p - N^\top N)^{-2} \nu$. Let $\mathcal {S} : = \{x+ iv: x>b+\ep\}$ for a sufficiently small $\ep>0$, and let us work on the set of full measure where $\|N^\top N\|<b+\ep/2$ eventually, and where $f_p(r)\to m(r)$. By inspection, $f_p$ are analytic functions on $\mathcal {S}$ bounded as $|f_p|\le 2/\ep$. Hence, by Lemma \ref{holo_derivative_conv}, $f_p'(r)\to m'(r)$.

In conclusion, $\nu^\top N^\top (t^2 I_p - N^\top N)^{-2} N \nu \to m(t^2) + t^2 m'(t^2)$, finishing the proof.

\subsection{Proof of Thm.\ \ref{spec} - Multispiked extension}
\label{pf_multi}


Let us denote by $\nu_i=M u_i/\xi_i^{1/2}$ the normalized backprojected signals, where $\xi_i = \|M  u_i\|^2 \to \tau_i$. For the proof we start as in Sec.\ \ref{spec_pf}, showing that we can equivalently work with $B_i^* =  \sum_{k=1}^r (\xi_k \ell_k)^{1/2} z_{ik} \nu_k +\ep^*_i.$  Defining the $r\times r$ diagonal matrices $L$, $\Delta$ with diagonal entries $\ell_k$, $\xi_k$ (respectively), and the $n \times r$, $p \times r$ matrices $Z ,\mathcal{V}$, with columns $Z_k = n^{-1/2}(z_{1k}, \ldots, z_{nk})^\top$ and $\nu_k$ respectively, we can thus work with 
$$\tilde B^* =  Z (\Delta L)^{1/2}\mathcal{V}^\top + N.$$
The matrix $M_n(t)$ is now $2r\times 2r$, and has the form
\begin{align*}M_n(t) = 
\begin{bmatrix} 
t \cdot Z^\top (t^2 I_n - NN^\top)^{-1} Z & 
Z^\top (t^2 I_n - NN^\top)^{-1} N\mathcal{V} \\ 
\mathcal{V}^\top N^\top (t^2 I_n - NN^\top)^{-1} Z & 
t \cdot \mathcal{V}^\top (t^2 I_p - N^\top N)^{-1} \mathcal{V}
\end{bmatrix} 
-  
\begin{bmatrix} 0_r & (\Delta L)^{-1/2} \\  (\Delta L)^{-1/2} & 0_r \end{bmatrix}.
\end{align*}
It is easy to see that Lemma \ref{noise_lem} still holds in this case. To find the limits of the entries of $M_n$, we need the following additional statement. 

\begin{lem}[Multispiked quadratic forms] 
\label{qf_lem3}
The quadratic forms in the multispiked case have the following properties for $t>b$:
\benum
\item $Z_k^\top R^\alpha Z_j \to 0$ a.s. for $\alpha=1,2$, if $k\neq j$.
\item $\nu_k^\top (t^2 I_p - N^\top N)^{-\alpha} \nu_j \to 0$ a.s. for $\alpha=1,2$, if $k\neq j$.
\eenum
\end{lem}
This lemma is proved in Sec.\ \ref{pf_qf_lem3}, using similar techniques as those in Lemma \ref{noise_lem}.  Defining the $r\times r$ diagonal matrices $T$ with diagonal entries $\tau_k$, we conclude that for $t>b$, $M_n(t)\to M(t)$ a.s., where now
$$M(t) = 
\begin{bmatrix} 
-t \cdot \underline m(t^2) I_r & 
-(T L)^{-1/2} \\ 
-(T L)^{-1/2} &
-t \cdot m(t^2) I_r
\end{bmatrix}.
$$
As before,  by Lemma A.1 of \cite{benaych2012singular}, we get that for $\tau_k\ell_k>1/D(b^2)$, $\sigma_k(\tilde B^*) \to t_k$ a.s., where $ t_k^2 \cdot \underline m(t_k^2) m(t_k^2)  = 1/(\tau_k\ell_k)$. This finishes the spike convergence proof.
 
To obtain the limit of the angles for $\hat u_k$ for a $k$ such that $\ell_k>\tau_k D(b^2)$, consider the left singular vectors $\hat Z_k$ associated to $\sigma_k(\tilde B^*)$. Define the $2r$-vector 
$$\alpha= 
\begin{bmatrix} 
 \beta_1\\
 \beta_2
\end{bmatrix}
= 
\begin{bmatrix} 
 (\Delta L)^{1/2}\mathcal{V}^\top \hat u_k \\ 
 (\Delta L)^{1/2}Z^\top \hat Z_k 
\end{bmatrix}.
$$
The vector $\alpha$ belongs to the kernel of $M_n(\sigma_k(\tilde B^*))$. As argued by \cite{benaych2012singular}, the fact that the projection of $\alpha$ into the orthogonal complement of $M(t_k)$ tends to zero, implies that $\alpha_j\to 0$ for all $j\notin \{k,k+r\}$. This proves that $\nu_j^\top \hat u_k \to 0$ for $j\neq k$, and the analogous claim for the left singular vectors. 

The linear equation $M_n(\sigma_k(\tilde B^*))\alpha=0$ in the $k$-th coordinate, where $k \le r$, reads (with $t=\sigma_k(\tilde B^*)$):
$$ t \alpha_k Z_k^\top R Z_k  - \alpha_{r+k} (\xi_k \ell_k)^{-1/2} 
+ \sum_{i\neq k} M_n(\sigma_k(\tilde B^*))_{ik} \alpha_k=0.$$
Only the first two terms are non-negligible due to the behavior of $M_n$, so we obtain $ t \alpha_k Z_k^\top R Z_k  = \alpha_{r+k} (\xi_k \ell_k)^{-1/2} +o_p(1)$. 
Moreover taking the norm of the equation $\hat Z_k = R( t Z \beta_1 + N\mathcal{V} \beta_2)$ (see Lemma 5.1 in \cite{benaych2012singular}), we get
$$t^2 \sum_{i,j \le r} \alpha_i \alpha_j  Z_i^\top R^2 Z_j 
+ \sum_{i,j \le r}  \alpha_{k+i} \alpha_{k+j}  \nu_i^\top N^\top R^2 N \nu_j 
+ \sum_{i,j \le r}  \alpha_i \alpha_{k+j}  Z_i R^2 N \nu_j =1.$$
From Lemma \ref{qf_lem3} and the discussion above, only the terms $\alpha_k^2 Z_k^\top R^2 Z_k$ and $\alpha_{r+k}^2 \nu_k^\top N^\top R^2 N \nu_k$ are non-negligible, so we obtain 
$$t^2\alpha_k^2 Z_k^\top R^2 Z_k + 
\alpha_{r+k}^2 \nu_k^\top N^\top R^2 N \nu_k =1 + o_p(1).$$
Combining the two equations above, 
$$\alpha_{r+k}^2\left[ \frac{Z_k^\top R^2 Z_k}{\xi_k \ell_k (Z_k^\top R Z_k)^2} + \nu_k^\top N^\top R^2 N \nu_k \right]=1 + o_p(1).$$
Since this is the same equation as in the single-spiked case, we can take the limit in a completely analogous way. This  finishes the proof.

\subsection{Proof of Lemma \ref{qf_lem3}}
\label{pf_qf_lem3}

\textbf{Part 1}: The convergence $Z_k^\top R^\alpha Z_j \to 0$ a.s. for $\alpha=1,2$, if $k\neq j$, follows directly from the following well-known lemma, cited from \cite{couillet2011random}:

\begin{lem}[Proposition 4.1 in \cite{couillet2011random}]
Let $x_n \in \mathbb{R}^n $ and $y_n \in \mathbb{R}^n $ be independent sequences of random vectors, such that for each $n$ the coordinates of $x_n$ and  $y_n$ are independent random variables. Moreover, suppose that the coordinates of $x_n$ are identically distributed with mean 0, variance $C/n$ for some $C>0$ and fourth moment of order $1/n^2$. Suppose the same conditions hold for $y_n$, where the distribution of the coordinates of $y_n$ can be different from those of $x_n$. Let $A_n$ be a sequence of $n \times n$ random matrices such that $\|A_n\|$ is uniformly bounded. Then $x_n^\top A_n y_n \to_{a.s.} 0$.
\end{lem}

\textbf{Part 2}: To show $\nu_k^\top (t^2 I_p - N^\top N)^{-\alpha} \nu_j \to 0$ a.s. for $\alpha=1,2$, if $k\neq j$, the same technique cannot be used, because the vectors $u_k$ are deterministic. However, it is straightforward to check that the method of \cite{bai2007asymptotics} that we adapted in proving Part 3 of Lemma \ref{qf_lem} extends to proving $\nu_k^\top (t^2 I_p - N^\top N)^{-1} \nu_j \to 0$. Indeed, it is easy to see that all their bounds hold unchanged. In the final step, as a deterministic equivalent for $\nu_k^\top (t^2 I_p - N^\top N)^{-1} \nu_j \to 0$,  one obtains $\nu_k^\top (t^2 I_p - P)^{-1} \nu_j$, where $P$ is the diagonal matrix with $(j,j)$-th entry $\Var{\ep_{ij}^*}$, and this bilinear form tends to 0 by our assumption, showing $\nu_k^\top (t^2 I_p - N^\top N)^{-1} \nu_j \to 0$. Then $\nu_k^\top (t^2 I_p - N^\top N)^{-2} \nu_j \to 0$ follows from the derivative trick employed in Part 3 of Lemma \ref{qf_lem2}. This finishes the proof.

\subsection{The uniform model and the simple form of BLP}
\label{unif_mod}

Here we introduce the \textit{uniform model}, a special case of the linearly transformed model. 
In the uniform model, we have that $\E A_i^\top A_i = m I_p$, and moreover that $\E \ep_i \ep_i^\top  = I_p$. 
This is useful for justifying the simpler form of the BLP that we are using:

\begin{prop}
\label{simple_blp}
In the uniform model, under the assumptions of Theorem \ref{spec}, suppose also that the first eight moments of  the entries of $A_i^\top A_i$ are uniformly bounded. Then the BLP $\hat X_i^{BLP}$ given in \eqref{blp} is asymptotically equivalent to $\hat X_i^0 = \sum_k \ell_k/(1+m\ell_k) \cdot u_k u_k^\top A_i^\top Y_i$, in the sense that $\E |\hat X_i^{BLP} - \hat X_i^0|^2 \to 0$.

\end{prop}

The proof is given below:

Recall that we observe $Y_i = A_iX_i+\ep_i$. The BLP has the form $\hat X_i = \Sigma_X A_i^\top(A_i \Sigma_X A_i^\top + \Sigma_\ep)^{-1} Y_i$. Now, $\Sigma_X = \sum_k \ell_k \cdot u_k u_k^\top$, while $\Sigma_\ep = I_p$ in the uniform case. With $G =  ( A_i \sum_k \ell_k \cdot u_k u_k^\top A_i^\top +I_p)^{-1}$, we get that $\hat X_i = \sum_k \ell_k \cdot u_k u_k^\top A_i^\top G Y_i$.

\subsection{Asymptotic BLP after backprojection.}

Our goal is to show that $\hat X_i$ is equivalent to $\hat X_i^0 = \sum_k \ell_k/(1+m\ell_k) \cdot u_k u_k^\top A_i^\top Y_i$, in the sense that $\E |\hat X_i - \hat X_i^0|^2 \to 0$.

Let us denote $v_k =A_i u_{k}$,  and $G_k = (\sum_{j\neq k}\ell_j  v_{j} v_{j}^\top+I)^{-1} $. 

Then $\hat X_i  = \sum_k\ell_k u_k v_k^\top G Y_{i}  $ and $\hat X_i^0  = \sum_k \ell_k/(1+m\ell_k)  \cdot  u_k v_k^\top Y_{i}$. Let us also define $m_k =v_{k}^\top GY_{i}  -1 /(1+m\ell_k) \cdot v_{k}^\top Y_i  $. Then, clearly,  
$$
\hat X_i - \hat X_i^0 = \sum_km_k\cdot\ell_k u_k 
$$
Therefore it is enough to show that $\E m_k^2 \to 0$. 

Using the rank one perturbation formula $u^\top (uu^\top +T)^{-1} = u^\top T^{-1}/(1+ u^\top T^{-1}u)$, we can write
$$
v_k^\top G = v_{k}^\top \left[\sum_{j}\ell_j  v_{j}  v_{j}^\top+I\right]^{-1}
=v_{k}^\top G_k 
/\left(1+  \ell_kv_{k}^\top G_k  v_{k}\right).
$$

In addition, by using the  formula $(VV^\top +I)^{-1} = [I- V (V^\top V +I)^{-1}V^\top]$ for $V=[\ell_1^{1/2} v_{1}, \ldots, \ell_r^{1/2} v_{r}]$ (excluding $v_{k}$), we conclude that 
$$
v_{k}^\top  G_k Y_{i} =
v_{k}^\top \left[\sum_{j\neq k}\ell_j  v_{j}  v_{j}^\top+ I \right]^{-1}Y_{i} 
=  v_{k}^\top Y_{i}  -  v_{k}^\top V (V^\top V +I)^{-1}V^\top Y_{i}.
$$

We thus have
\begin{align*}
m_k &
= \frac{v_{k}^\top G_k Y_{i}}{1+\ell_k v_{k}^\top G_kv_{k}}  
-  \frac{v_{k}^\top Y_{i}}{1+m \ell_k} \\
& = v_{k}^\top Y_{i} \left(\frac{1}{1+\ell_k v_{k}^\top G_kv_{k}}-\frac{1}{1+m \cdot\ell_k}\right)
- \frac{v_{k}^\top V (V^\top V +I)^{-1}V^\top Y_{i}}{1+\ell_k v_{k}^\top G_kv_{k}}.
\end{align*}

Thus, it is enough to show
\benum
\item $\E (v_{k}^\top Y_i)^2\cdot (v_{k}^\top G_kv_{k}-m)^2 \to0$
\item $\E [v_{k}^\top V (V^\top V +I)^{-1}V^\top Y_i]^2 \to 0$
\eenum

We prove these in turn below:

\benum

\item

First, for (1): 

By using the  formula $(VV^\top +I)^{-1} =[I- V (V^\top V + I)^{-1}V^\top]$, we see
$$
v_k^\top G_k v_k =  v_{k}^\top \left[\sum_{j\neq k}\ell_j  v_{j}  v_{j}^\top+I \right]^{-1}v_{k} 
=  v_{k}^\top v_{k} -  v_{k}^\top V (V^\top V + I)^{-1}V^\top v_{k}.
$$
But $x^\top (V^\top V + I)^{-1} x \le x^\top x$, because the eigenvalues of $(V^\top V + I)^{-1}$ are all at most unity. 
Thus, it is enough to show $\E (v_{k}^\top Y_i)^2\cdot |V^\top v_k|^4 \to0$ and $\E (v_{k}^\top Y_i)^2\cdot (v_k^\top v_k -m)^2 \to0$. For this, it is enough to show that $\E |V^\top v_k|^8 \to0$, $\E (v_{k}^\top Y_i)^4 \to0$, and $\E(v_k^\top v_k -m)^4 \to0$.

First we show $\E |V^\top v_k|^8 \to0$. The entries of the $r-1$-dimensional vector $V^\top v_{k}$ are  $ \ell_j^{1/2}v_{j}^\top v_{k} =  \ell_j^{1/2}u_{j}^\top D_i u_{k}$ for $j\neq k$. 
But we have that $\E (u_{j}^\top D_i u_{k})^8 \to 0$. Indeed, since $D_i = m I +E_i$, and $u_k^\top u_j\to 0$, we only need to show that  $\E (u_{j}^\top E_i u_{k})^8 \to 0.$ Since we assumed that $E_{ij}$ has bounded 8th moments, this follows by expanding the moment. The details are omitted for brevity. Thus, under the assumptions made on $E_i$, $\E|V^\top v_k|^8\to0$.

Second, we show that  $\E (v_{k}^\top Y_i)^8 \to0$. Indeed, $v_{k}^\top Y_i=u_k^\top A_i^\top (A_i\sum_j \ell_j^{1/2} z_{ij} u_j +\ep_i)$. We can take the expectation over $z_{ij}$ and $\ep_i$, because they are independent from $v_k$. We obtain $\E (v_{k}^\top Y_i)^{8} =  \sum_j \ell_j^4 \E (u_k^\top D_i u_j)^{8}+\E (u_k^\top D_i u_k)^4$. Similarly to above, each term converges to zero. 

Third, and finally, we show $\E(v_k^\top v_k -m)^4 \to0$. Indeed, $v_k^\top v_k -m = u_{k}^\top D_i u_{k} - m \cdot u_k^\top u_k = u_{k}^\top (D_i  - mI) u_{k}$, so 
 the claim is equivalent to $\E (u_k^\top E_i u_k)^4 \to 0$, which follows as above. 
This finishes the proof.

\item 

By using $x^\top (V^\top V + I)^{-1} x \le x^\top x$, it is enough to show $\E [v_{k}^\top V V^\top Y_i]^2 \to 0$. As above, this can be achieved by taking the expectation over $z_{ij}$ and $\ep_i$ first, and then controlling $v_{k}^\top V$. The details are omitted for brevity. 
\eenum

This finishes the proof that BLP is equivalent to simple linear denoising.

\subsection{Proof of Lemma \ref{hat_M}}
\label{hat_M_pf}
We will show that the operator norm of $\hat{M}^{-1} - M^{-1}$ converges to 0; since the operator norm of the data matrix $B$ converges almost surely by the main theorem, the result follows. This is equivalent to showing that $\sup_i|M_i^{-1} - \hat{M}_i^{-1}| \to 0$, which, since the $M_i$ are uniformly bounded away from 0, will follow if we show $\sup_i|M_i - \hat{M}_i| \to 0$ almost surely.

To show this, observe that by the Central Absolute Moment Inequality \citep[see, for example,][]{mukhopadhyay2000probability} and the moment condition on $D_{ij}$, there is an absolute constant $C$ such that
$$
    \mathbb{E}|\hat{M}_i - M_i|^{4+\phi} \le C n^{-(2+\phi/2)}
$$

for all $i=1,\dots,p$. Therefore, for any $a > 0$,
\begin{align*}
    \text{Pr}(\sup_i|\hat{M}_i - M_i|  \ge a) 
    & \le \sum_{i=1}^n \text{Pr}(|\hat{M}_i - M_i|  \ge a)
        \nonumber \\
    & \le C \cdot n \cdot \frac{\mathbb{E}|\hat{M}_i - M_i|^{4+\phi}}{a^{4+\phi}}
        \nonumber \\
    & \le C a^{-(4+\phi)} n^{-(1+\phi/2)}.
\end{align*}
Since this is summable, it follows that $\sup_i|\hat{M}_i - M_i| \to 0$ almost surely, as desired.

\subsection{Derivation of optimal singular values and AMSE}
\label{sv-derived}
The derivation of the optimal singualar values and the AMSE is a summary of what is found in \cite{gavish-donoho-2017}. We provide it here for the reader's convenience.

\begin{prop}
There exist orthonormal bases $a_1,\dots,a_p \in \mathbb{R}^p$ and $a_1^\prime,\dots,a_n^\prime \in \mathbb{R}^n$ in which $n^{-1/2}X$ and $n^{-1/2}\hat{X}$ are jointly in block-diagonal form, with $r$ 2-by-2 blocks. More precisely, there are orthogonal matrices $A_1$ and $A_2$ such that $A_1 \hat X A_2^\top = \oplus_{i=1}^r C_i$ and $A_1 X A_2^\top = \oplus_{i=1}^r D_i$, where $C_i$ and $D_i$ are 2-by-2 matrices given by
\begin{align*}
    C_i =         
    \left(
    \begin{array}{c c}
    \lambda_i  c_i \tilde{c}_i & \lambda_i  c_i \tilde{s_i}     \\
    \lambda_i  \tilde{c}_i s_i& \lambda_i  s_i\tilde{s}_i        
    \end{array}
    \right), \,\,\,\,\,\,\,\,\,
    D_i = 
    \left(
    \begin{array}{c c}
    \ell_i^{1/2} &  0\\
      0         &  0
    \end{array}
    \right).
\end{align*}
\end{prop}
\begin{proof}
The proof, which is elementary linear algebra, is essentially contained in \cite{gavish-donoho-2017}.
\end{proof}

Since the rank $r$ is fixed, and since the sines and cosines converge almost surely, it follows immediately that the quantity 
$$
    \L_\infty(\lambda_1,\dots,\lambda_r) = \lim_{p,n\to\infty} n^{-1}\|\hat{X} - X \|_F^2
$$
is well-defined, where $\hat{X}$ is the estimator such that $n^{-1} \hat{X}^\top\hat{X}$ has eigenvalues $\lambda_1,\dots,\lambda_r$; and furthermore, since the squared Frobenius norm decomposes over blocks, we have:
$$
    \L_\infty(\lambda_1,\dots,\lambda_r) = \sum_{i=1}^r \|C_i - D_i \|_F^2
$$
Consequently, the optimal $\lambda_i$ is found by optimizing a single spike for a 2-by-2 block:
\begin{align}
\label{eq-tstar}
    \lambda_i^* = \operatorname*{arg\,min}_\lambda \|C_i - D_i\|_F^2.
\end{align}

To solve for $\lambda_i^*$ and find the AMSE, we write out the error explicitly:
\begin{align*}
\|C_i - D_i\|_F^2 = \lambda_i^2 - 2 \ell_i^{1/2} c_i \tilde{c}_i + \ell_i
\end{align*}
which is minimized at $\lambda_i^* = \ell_i^{1/2}c_i \tilde{c}_i$, and has minimum value $\ell_i(1 - c_i \tilde{c}_i)$. The total AMSE is therefore $\sum_{i=1}^r\ell_i(1 - c_i \tilde{c}_i)$.

\subsection{Proof of Thm.\ \ref{oos-prop}}
\label{oos-proof}

All that remains to show is that the out-of-sample $W$-AMSE is equal to the in-sample $W$-AMSE, when $W = \Sigma_{\ep}^{-1/2}$. That is, we must show
\begin{align*}
    \sum_{i=1}^r 
        \left( \tilde{\ell}_i
        - \frac{\tilde{\ell}_i^2 c_i^4}{\tilde{\ell}_i c_i^2 + 1} \right)
    = \sum_{k=1}^r \tilde{\ell}_i (1 - c_i^2 \tilde{c}_i^2).
\end{align*}
We will prove equality of the individual summands; denoting $\ell = \tilde{\ell}_i$, $c = c_i$ and $\tilde{c} = \tilde{c}_i$, this means showing 
\begin{math}
    \ell - \frac{ \ell^2 c_k^4}{\ell c^2 + 1} 
    =  \ell (1 - c^2 \tilde{c}^2).
\end{math}
Straightforward algebraic manipulation shows that this is equivalent to showing
\begin{math}
    1 / \tilde{c}^2 = 1 + 1 / (\ell c^2).
\end{math}
Substituting formulas \eqref{cos_inner} and \eqref{cos_outer}, we have:
\begin{align*}
    1 / \tilde{c}^2 = \frac{1 + 1/\ell}{1 - \gamma / \ell^2}
                    = \frac{\ell + 1}{\ell - \gamma/\ell}
                    = 1 + \frac{1 + \gamma/\ell}{\ell - \gamma/\ell}
                    = 1 + \frac{1}{\ell}\frac{1 + \gamma/\ell}{1 - \gamma/\ell^2}
    = 1 + 1 / (\ell c^2)
\end{align*}
as desired.

{\small
\setlength{\bibsep}{0.0pt plus 0.0ex}
\bibliographystyle{plainnat-abbrev}
\bibliography{references}

\begin{thebibliography}{66}
\providecommand{\natexlab}[1]{#1}
\providecommand{\url}[1]{\texttt{#1}}
\expandafter\ifx\csname urlstyle\endcsname\relax
  \providecommand{\doi}[1]{doi: #1}\else
  \providecommand{\doi}{doi: \begingroup \urlstyle{rm}\Url}\fi

\bibitem[And{\'e}n et~al.(2015)And{\'e}n, Katsevich, and
  Singer]{anden2015covariance}
J.~And{\'e}n, E.~Katsevich, and A.~Singer.
\newblock Covariance estimation using conjugate gradient for {3D}
  classification in cryo-{EM}.
\newblock In \emph{Biomedical Imaging (ISBI), 2015 IEEE 12th International
  Symposium on}, pages 200--204. IEEE, 2015.

\bibitem[{ASPIRE}(2017)]{aspire}
{ASPIRE}.
\newblock {Algorithms for Single Particle Reconstruction}.
\newblock \url{http://spr.math.princeton.edu/}, 2017.

\bibitem[Bai et~al.(2015)Bai, McMullan, and Scheres]{bai2015cryo}
X.-C. Bai, G.~McMullan, and S.~H. Scheres.
\newblock How cryo-{EM} is revolutionizing structural biology.
\newblock \emph{Trends in Biochemical Sciences}, 40\penalty0 (1):\penalty0
  49--57, 2015.

\bibitem[Bai et~al.(2007)Bai, Miao, and Pan]{bai2007asymptotics}
Z.~Bai, B.~Miao, and G.~Pan.
\newblock On asymptotics of eigenvectors of large sample covariance matrix.
\newblock \emph{The Annals of Probability}, 35\penalty0 (4):\penalty0
  1532--1572, 2007.

\bibitem[Bai and Ding(2012)]{bai2012estimation}
Z.~Bai and X.~Ding.
\newblock Estimation of spiked eigenvalues in spiked models.
\newblock \emph{Random Matrices: Theory and Applications}, 1\penalty0
  (02):\penalty0 1150011, 2012.

\bibitem[Bai and Silverstein(2009)]{bai2009spectral}
Z.~Bai and J.~W. Silverstein.
\newblock \emph{Spectral analysis of large dimensional random matrices}.
\newblock Springer Series in Statistics. Springer, 2009.

\bibitem[Bai and Yao(2012)]{bai2012sample}
Z.~Bai and J.~Yao.
\newblock On sample eigenvalues in a generalized spiked population model.
\newblock \emph{Journal of Multivariate Analysis}, 106:\penalty0 167--177,
  2012.

\bibitem[Baik and Silverstein(2006)]{baik2006eigenvalues}
J.~Baik and J.~W. Silverstein.
\newblock Eigenvalues of large sample covariance matrices of spiked population
  models.
\newblock \emph{Journal of Multivariate Analysis}, 97\penalty0 (6):\penalty0
  1382--1408, 2006.

\bibitem[Baik et~al.(2005)Baik, Ben~Arous, and P{\'e}ch{\'e}]{baik2005phase}
J.~Baik, G.~Ben~Arous, and S.~P{\'e}ch{\'e}.
\newblock Phase transition of the largest eigenvalue for nonnull complex sample
  covariance matrices.
\newblock \emph{Annals of Probability}, 33\penalty0 (5):\penalty0 1643--1697,
  2005.

\bibitem[Benaych-Georges and Nadakuditi(2012)]{benaych2012singular}
F.~Benaych-Georges and R.~R. Nadakuditi.
\newblock The singular values and vectors of low rank perturbations of large
  rectangular random matrices.
\newblock \emph{Journal of Multivariate Analysis}, 111:\penalty0 120--135,
  2012.

\bibitem[Bhamre et~al.(2016)Bhamre, Zhang, and Singer]{bhamre2016denoising}
T.~Bhamre, T.~Zhang, and A.~Singer.
\newblock Denoising and covariance estimation of single particle cryo-{EM}
  images.
\newblock \emph{Journal of Structural Biology}, 195\penalty0 (1):\penalty0
  72--81, 2016.

\bibitem[Blackledge(2006)]{blackledge2006digital}
J.~M. Blackledge.
\newblock \emph{Digital Signal Processing: Mathematical and Computational
  Methods, Software Development and Applications}.
\newblock Elsevier, 2006.

\bibitem[Buja and Eyuboglu(1992)]{buja:eyob:1992}
A.~Buja and N.~Eyuboglu.
\newblock Remarks on parallel analysis.
\newblock \emph{Multivariate Behavioral Research}, 27\penalty0 (4):\penalty0
  509--540, 1992.

\bibitem[Callaway(2015)]{callaway2015revolution}
E.~Callaway.
\newblock The revolution will not be crystallized.
\newblock \emph{Nature}, 525\penalty0 (7568):\penalty0 172, 2015.

\bibitem[Campisi and Egiazarian(2016)]{campisi2016blind}
P.~Campisi and K.~Egiazarian.
\newblock \emph{Blind image deconvolution: theory and applications}.
\newblock CRC press, 2016.

\bibitem[Cand\`{e}s and Plan(2010)]{candes2010noise}
E.~J. Cand\`{e}s and Y.~Plan.
\newblock Matrix completion with noise.
\newblock \emph{Proceedings of the IEEE}, 98\penalty0 (6):\penalty0 925--936,
  2010.

\bibitem[Cand{\`e}s and Recht(2009)]{candes2009exact}
E.~J. Cand{\`e}s and B.~Recht.
\newblock Exact matrix completion via convex optimization.
\newblock \emph{Foundations of Computational Mathematics}, 9\penalty0
  (6):\penalty0 717--772, 2009.

\bibitem[Cand\`{e}s and Tao(2010)]{candes2010power}
E.~J. Cand\`{e}s and T.~Tao.
\newblock The power of convex relaxation: Near-optimal matrix completion.
\newblock \emph{IEEE Transactions on Information Theory}, 56\penalty0
  (5):\penalty0 2053--2080, 2010.

\bibitem[Chen et~al.(2015)Chen, Bhojanapalli, Sanghavi, and
  Ward]{chen2015completing}
Y.~Chen, S.~Bhojanapalli, S.~Sanghavi, and R.~Ward.
\newblock Completing any low-rank matrix, provably.
\newblock \emph{The Journal of Machine Learning Research}, 16\penalty0
  (1):\penalty0 2999--2034, 2015.

\bibitem[Couillet and Debbah(2011)]{couillet2011random}
R.~Couillet and M.~Debbah.
\newblock \emph{Random {M}atrix {M}ethods for {W}ireless {C}ommunications}.
\newblock Cambridge University Press, 2011.

\bibitem[Dobriban(2015)]{dobriban2015efficient}
E.~Dobriban.
\newblock Efficient computation of limit spectra of sample covariance matrices.
\newblock \emph{Random Matrices: Theory and Applications}, 04\penalty0
  (04):\penalty0 1550019, 2015.

\bibitem[Dobriban(2017)]{dobriban2017factor}
E.~Dobriban.
\newblock Factor selection by permutation.
\newblock \emph{arXiv preprint arXiv:1710.00479}, 2017.

\bibitem[Dobriban and Owen(2017)]{dobriban2017deterministic}
E.~Dobriban and A.~B. Owen.
\newblock Deterministic parallel analysis.
\newblock \emph{arXiv preprint arXiv:1711.04155}, 2017.

\bibitem[Dobriban and Wager(2015)]{dobriban2015high}
E.~Dobriban and S.~Wager.
\newblock High-dimensional asymptotics of prediction: Ridge regression and
  classification.
\newblock \emph{arXiv preprint arXiv:1507.03003}, 2015.

\bibitem[Donoho et~al.(2013)Donoho, Gavish, and Johnstone]{donoho2013optimal}
D.~L. Donoho, M.~Gavish, and I.~M. Johnstone.
\newblock Optimal shrinkage of eigenvalues in the spiked covariance model.
\newblock \emph{arXiv preprint arXiv:1311.0851, to appear in AoS}, 2013.

\bibitem[Efron(2012)]{efron2012large}
B.~Efron.
\newblock \emph{Large-scale inference: empirical Bayes methods for estimation,
  testing, and prediction}, volume~1.
\newblock Cambridge University Press, 2012.

\bibitem[Erd{\H{o}}s and Yau(2012)]{erdos2012universality}
L.~Erd{\H{o}}s and H.-T. Yau.
\newblock Universality of local spectral statistics of random matrices.
\newblock \emph{Bulletin of the American Mathematical Society}, 49\penalty0
  (3):\penalty0 377--414, 2012.

\bibitem[Gavish and Donoho(2014{\natexlab{a}})]{donoho2014minimax}
M.~Gavish and D.~L. Donoho.
\newblock Minimax risk of matrix denoising by singular value thresholding.
\newblock \emph{The Annals of Statistics}, 42\penalty0 (6):\penalty0
  2413--2440, 2014{\natexlab{a}}.

\bibitem[Gavish and Donoho(2017)]{gavish-donoho-2017}
M.~Gavish and D.~L. Donoho.
\newblock Optimal shrinkage of singular values.
\newblock \emph{IEEE Transactions on Information Theory}, 63\penalty0
  (4):\penalty0 2137--2152, 2017.

\bibitem[Gavish and Donoho(2014{\natexlab{b}})]{gavish2014optimal}
M.~Gavish and D.~L. Donoho.
\newblock Optimal shrinkage of singular values.
\newblock \emph{arXiv preprint arXiv:1405.7511}, 2014{\natexlab{b}}.

\bibitem[Golub and {Van Loan}(2012)]{golub2012matrix}
G.~H. Golub and C.~F. {Van Loan}.
\newblock \emph{Matrix Computations}, volume~3.
\newblock JHU Press, 2012.

\bibitem[Hachem et~al.(2015)Hachem, Hardy, and Najim]{hachem2015survey}
W.~Hachem, A.~Hardy, and J.~Najim.
\newblock A survey on the eigenvalues local behavior of large complex
  correlated wishart matrices.
\newblock \emph{ESAIM: Proceedings and Surveys}, 51:\penalty0 150--174, 2015.

\bibitem[Jain et~al.(2013)Jain, Netrapalli, and Sanghavi]{jain2013}
P.~Jain, P.~Netrapalli, and S.~Sanghavi.
\newblock Low-rank matrix completion using alternating minimization.
\newblock In \emph{Proceedings of the Forty-fifth Annual ACM Symposium on
  Theory of Computing}, pages 665--674. ACM, 2013.

\bibitem[Ji and Ye(2009)]{ji2009agm}
S.~Ji and J.~Ye.
\newblock An accelerated gradient method for trace norm minimization.
\newblock In \emph{Proceedings of the 26th Annual International Conference on
  Machine Learning}, pages 457--464. ACM, June 2009.

\bibitem[Johnstone(2001)]{johnstone2001distribution}
I.~M. Johnstone.
\newblock On the distribution of the largest eigenvalue in principal components
  analysis.
\newblock \emph{Annals of Statistics}, 29\penalty0 (2):\penalty0 295--327,
  2001.

\bibitem[Johnstone and Onatski(2015)]{johnstone2015testing}
I.~M. Johnstone and A.~Onatski.
\newblock Testing in high-dimensional spiked models.
\newblock \emph{arXiv preprint arXiv:1509.07269}, 2015.

\bibitem[Kam(1980)]{Kam1980}
Z.~Kam.
\newblock {The reconstruction of structure from electron micrographs of
  randomly oriented particles.}
\newblock \emph{Journal of Theoretical Biology}, 82\penalty0 (1):\penalty0
  15--39, 1980.

\bibitem[Katsevich et~al.(2015)Katsevich, Katsevich, and
  Singer]{katsevich2015covariance}
E.~Katsevich, A.~Katsevich, and A.~Singer.
\newblock Covariance matrix estimation for the cryo-{EM} heterogeneity problem.
\newblock \emph{SIAM Journal on Imaging Sciences}, 8\penalty0 (1):\penalty0
  126--185, 2015.

\bibitem[Keshavan and Montanari(2010)]{keshavan2010regularization}
R.~H. Keshavan and A.~Montanari.
\newblock Regularization for matrix completion.
\newblock In \emph{Proceedings of International Symposium on Information
  Theory}, pages 1503--1507. IEEE, 2010.

\bibitem[Keshavan et~al.(2009)Keshavan, Oh, and Montanari]{keshavan2009matrix}
R.~H. Keshavan, S.~Oh, and A.~Montanari.
\newblock Matrix completion from a few entries.
\newblock In \emph{2009 IEEE International Symposium on Information Theory},
  pages 324--328. IEEE, 2009.

\bibitem[Keshavan et~al.(2010)Keshavan, Montanari, and Oh]{keshavan2010matrix}
R.~H. Keshavan, A.~Montanari, and S.~Oh.
\newblock Matrix completion from a few entries.
\newblock \emph{IEEE Transactions on Information Theory}, 56\penalty0
  (6):\penalty0 2980--2998, 2010.

\bibitem[Klopp(2014)]{klopp2014noisy}
O.~Klopp.
\newblock Noisy low-rank matrix completion with general sampling distribution.
\newblock \emph{Bernoulli}, 20\penalty0 (1):\penalty0 282--303, 2014.

\bibitem[Koltchinskii et~al.(2011)Koltchinskii, Lounici, and
  Tsybakov]{koltchinskii2011nuclear}
V.~Koltchinskii, K.~Lounici, and A.~B. Tsybakov.
\newblock Nuclear-norm penalization and optimal rates for noisy low-rank matrix
  completion.
\newblock \emph{The Annals of Statistics}, pages 2302--2329, 2011.

\bibitem[Kritchman and Nadler(2008)]{kritchman2008determining}
S.~Kritchman and B.~Nadler.
\newblock Determining the number of components in a factor model from limited
  noisy data.
\newblock \emph{Chemometrics and Intelligent Laboratory Systems}, 94\penalty0
  (1):\penalty0 19--32, 2008.

\bibitem[Mallat(2008)]{mallat2008wavelet}
S.~Mallat.
\newblock \emph{A Wavelet Tour of Signal Processing: The Sparse Way}.
\newblock Academic Press, 2008.

\bibitem[Marchenko and Pastur(1967)]{marchenko1967distribution}
V.~A. Marchenko and L.~A. Pastur.
\newblock Distribution of eigenvalues for some sets of random matrices.
\newblock \emph{Mat. Sb.}, 114\penalty0 (4):\penalty0 507--536, 1967.

\bibitem[Mukhopadhyay(2000)]{mukhopadhyay2000probability}
N.~Mukhopadhyay.
\newblock \emph{Probability and Statistical Inference}.
\newblock CRC Press, 2000.

\bibitem[Nadakuditi(2014)]{nadakuditi2014optshrink}
R.~R. Nadakuditi.
\newblock Optshrink: An algorithm for improved low-rank signal matrix denoising
  by optimal, data-driven singular value shrinkage.
\newblock \emph{IEEE Transactions on Information Theory}, 60\penalty0
  (5):\penalty0 3002--3018, 2014.

\bibitem[Nadakuditi and Edelman(2008)]{nadakuditi2008sample}
R.~R. Nadakuditi and A.~Edelman.
\newblock Sample eigenvalue based detection of high-dimensional signals in
  white noise using relatively few samples.
\newblock \emph{Signal Processing, IEEE Transactions on}, 56\penalty0
  (7):\penalty0 2625--2638, 2008.

\bibitem[Nadler(2008)]{nadler2008finite}
B.~Nadler.
\newblock Finite sample approximation results for principal component analysis:
  A matrix perturbation approach.
\newblock \emph{The Annals of Statistics}, 36\penalty0 (6):\penalty0
  2791--2817, 2008.

\bibitem[Negahban and Wainwright(2011)]{negahban2011estimation}
S.~Negahban and M.~J. Wainwright.
\newblock Estimation of (near) low-rank matrices with noise and
  high-dimensional scaling.
\newblock \emph{The Annals of Statistics}, pages 1069--1097, 2011.

\bibitem[Onatski(2012)]{onatski2012asymptotics}
A.~Onatski.
\newblock Asymptotics of the principal components estimator of large factor
  models with weakly influential factors.
\newblock \emph{Journal of Econometrics}, 168\penalty0 (2):\penalty0 244--258,
  2012.

\bibitem[Onatski et~al.(2013)Onatski, Moreira, and
  Hallin]{onatski2013asymptotic}
A.~Onatski, M.~J. Moreira, and M.~Hallin.
\newblock Asymptotic power of sphericity tests for high-dimensional data.
\newblock \emph{The Annals of Statistics}, 41\penalty0 (3):\penalty0
  1204--1231, 2013.

\bibitem[Onatski et~al.(2014)Onatski, Moreira, and Hallin]{onatski2014signal}
A.~Onatski, M.~J. Moreira, and M.~Hallin.
\newblock Signal detection in high dimension: The multispiked case.
\newblock \emph{The Annals of Statistics}, 42\penalty0 (1):\penalty0 225--254,
  2014.

\bibitem[Passemier and Yao(2012)]{passemier2012determining}
D.~Passemier and J.-F. Yao.
\newblock On determining the number of spikes in a high-dimensional spiked
  population model.
\newblock \emph{Random Matrices: Theory and Applications}, 1\penalty0
  (01):\penalty0 1150002, 2012.

\bibitem[Paul(2007)]{paul2007asymptotics}
D.~Paul.
\newblock Asymptotics of sample eigenstructure for a large dimensional spiked
  covariance model.
\newblock \emph{Statistica Sinica}, 17\penalty0 (4):\penalty0 1617--1642, 2007.

\bibitem[Paul and Aue(2014)]{paul2014random}
D.~Paul and A.~Aue.
\newblock Random matrix theory in statistics: A review.
\newblock \emph{Journal of Statistical Planning and Inference}, 150:\penalty0
  1--29, 2014.

\bibitem[Recht(2011)]{recht2011simpler}
B.~Recht.
\newblock A simpler approach to matrix completion.
\newblock \emph{Journal of Machine Learning Research}, 12:\penalty0 3413--3430,
  December 2011.

\bibitem[Rohde et~al.(2011)Rohde, Tsybakov, et~al.]{rohde2011estimation}
A.~Rohde, A.~B. Tsybakov, et~al.
\newblock Estimation of high-dimensional low-rank matrices.
\newblock \emph{The Annals of Statistics}, 39\penalty0 (2):\penalty0 887--930,
  2011.

\bibitem[Searle et~al.(2009)Searle, Casella, and McCulloch]{searle2009variance}
S.~R. Searle, G.~Casella, and C.~E. McCulloch.
\newblock \emph{Variance Components}, volume 391.
\newblock John Wiley \& Sons, 2009.

\bibitem[Singer and Wu(2013)]{singer2013two}
A.~Singer and H.-T. Wu.
\newblock Two-dimensional tomography from noisy projections taken at unknown
  random directions.
\newblock \emph{SIAM Journal on Imaging Sciences}, 6\penalty0 (1):\penalty0
  136--175, 2013.

\bibitem[Srebro and Salakhutdinov(2010)]{srebro2010collaborative}
N.~Srebro and R.~R. Salakhutdinov.
\newblock Collaborative filtering in a non-uniform world: {L}earning with the
  weighted trace norm.
\newblock In \emph{Advances in Neural Information Processing Systems}, pages
  2056--2064, 2010.

\bibitem[Stein and Shakarchi(2011)]{stein2011fourier}
E.~M. Stein and R.~Shakarchi.
\newblock \emph{Fourier Analysis: An Introduction}, volume~1.
\newblock Princeton University Press, 2011.

\bibitem[Vershynin(2010)]{vershynin2010introduction}
R.~Vershynin.
\newblock Introduction to the non-asymptotic analysis of random matrices.
\newblock \emph{arXiv preprint arXiv:1011.3027}, 2010.

\bibitem[Yao et~al.(2015)Yao, Bai, and Zheng]{yao2015large}
J.~Yao, Z.~Bai, and S.~Zheng.
\newblock \emph{Large Sample Covariance Matrices and High-Dimensional Data
  Analysis}.
\newblock Cambridge University Press, 2015.

\bibitem[Zhao et~al.(2016)Zhao, Shkolnisky, and Singer]{zhao2016fast}
Z.~Zhao, Y.~Shkolnisky, and A.~Singer.
\newblock Fast steerable principal component analysis.
\newblock \emph{IEEE Transactions on Computational Imaging}, 2\penalty0
  (1):\penalty0 1--12, 2016.

\end{thebibliography}
}

\end{document}